\documentclass[reqno, a4paper,12pt]{amsart}

\usepackage[capposition=top]{floatrow}
\usepackage{euscript,eufrak,verbatim}
\usepackage{graphicx}
\usepackage[usenames]{color}
\usepackage[colorlinks,linkcolor=red,anchorcolor=blue,citecolor=blue]{hyperref}
\usepackage{amsmath, mathtools}
\usepackage{mathtools}
\usepackage{amsthm}
\usepackage[all]{xy}
\usepackage{amssymb}
\usepackage{bm}

\usepackage[all]{xy}

\usepackage{mathrsfs}
\usepackage{amscd}

\usepackage{enumitem}

\makeatletter
\numberwithin{equation}{section}

\theoremstyle{plain}
\newtheorem{main theorem}{Main Theorem}
\newtheorem{theorem}{Theorem}[section]
\newtheorem{lemma}[theorem]{Lemma}

\newtheorem{corollary}[theorem]{Corollary}
\newtheorem{proposition}[theorem]{Proposition}
\newtheorem{claim}[theorem]{Claim}

\theoremstyle{definition}
\newtheorem{definition}[theorem]{Definition}
\newtheorem{remark}[theorem]{Remark}

\newcommand{\diam}{\mathrm{diam}}

\begin{document}
\title{Variational principles of relative weighted topological pressure}

\author{Zhengyu Yin*}



\maketitle

\begin{abstract}
Recently, M. Tsukamoto \cite{Tsu} (New approach to weighted topological entropy and pressure, Ergod. Theory Dyn. Syst. 43 (2023), 1004–1034) introduced a new approach to defining weighted topological entropy and pressure. Inspired by the ideas in \cite{Tsu}, we define the relative weighted topological entropy and pressure for factor maps and establish several variational principles. One of these results addresses a question raised by D. Feng and W. Huang \cite{FH} (Variational principle for weighted topological pressure, J. Math. Pures Appl. 106 (2016), 411–452), namely, whether there exists a relative version of the weighted variational principle. In this paper, we aim to establish such a variational principle.
Furthermore, we generalize the Ledrappier and Walters type relative variational principle to the weighted version.

\end{abstract}

\section{introduction}
Let $(X,T)$ be a topological dynamical system (TDS) with $X$ being a compact metric space. Given $f$ a continuous real-valued map on $X$, the well-known notion of topological pressure $P(T,f)$ which is a generalization of topological entropy in \cite{Akm} was introduced by D. Ruelle \cite{Rue} in 1973 and was extended by P. Walters \cite{Pw1} to compact spaces with continuous  transformation, and the variational principle was obtained by
\[P(T,f)=\sup\left(h_\mu(T)+\int fd\mu\right),\]
where the supremum is taken over all $T$-invariant Borel probability measures on $X$ endowed with the weak* topology and $h_\mu(T)$ is the measure-theoretical entropy of $\mu$.

Given TDSs $(X, T)$ and $(Y, S)$, we say that $Y$ is a factor of $X$ if there exists a surjective continuous map $\pi: X \to Y$ such that $\pi \circ T = S \circ \pi$. Let $\pi: (X,T) \to (Y,S)$ be a factor map, and let $f$ be a real-valued continuous map on $X$. In \cite{Lp}, F. Ledrappier and P. Walters introduced the notion of relative pressure, which extends the concept of topological pressure, and they proved the following relative variational principle:
\[
\int_Y P(T, \pi^{-1}(y), f) \, d\nu(y) = \sup \left( h_\mu(T \mid S) + \int_X f \, d\mu \right),
\]
where $\nu$ is an $S$-invariant measure on $Y$, and the supremum is taken over all $T$-invariant measures $\mu$ with $\nu = \pi\mu$. This is also referred to as the "Inner Variational Principle" in \cite{DS, Dow}. 

In 2002, T. Downarowicz and J. Serafin \cite{DS} studied fiber entropy and conditional entropy on non-metrizable spaces, obtaining more general variational principles related to these notions. Furthermore, A. Dooley and G. Zhang \cite{DZ} studied the notion of topological fiber entropy and conditional entropy for random dynamical systems over an infinite, countable, discrete amenable group. K. Yan \cite{YK} also explored related topics for general discrete countable amenable group actions, extending classical variational principles in these settings.

In addition to the variational principle, many other topics regarding the relative case for a factor map $\pi$ have been explored. For example, in \cite{GZ}, G. Zhang studied positive conditional entropy and chaos. In \cite{HYZ}, the same author, along with W. Huang and X. Ye, investigated local entropy concerning a factor map, obtaining a local version of the relative variational principle. They also studied the relative entropy tuple, relative C.P.E. extension, and relative U.P.E. extension in \cite{HYZ1}.

Given factor maps $\pi_i: X_i \to X_{i+1}$ for $i = 1, \dots, k$ between TDSs. Motivated by the fractal geometry of self-affine carpets and sponges \cite{Bed, KP, Mcm}, D. Feng and W. Huang \cite{FH} introduced the notion of weighted topological pressure for these factor maps and proved a corresponding variational principle. 
For example,
consider the case of a factor map $\pi: (X,T) \to (Y,S)$, where $\boldsymbol{a} = (a_1, a_2) \in \mathbb{R}^2$ with $a_1 > 0$ and $a_2 \geq 0$. Specifically, they defined the $\boldsymbol{a}$-weighted topological pressure $\mathcal{P}^{\boldsymbol{a}}(T,f)$ for a continuous map $f$ on $X$ with respect to $\pi$, and obtained the following formula:
\begin{align}\label{orithm}
    \mathcal{P}^{\boldsymbol{a}}(T,f) = \sup \left( a_1 h_\mu(X,T) + a_2 h_{\pi\mu}(Y,S) + \int_X f \, d\mu \right),
\end{align}
where the supremum is taken over all $T$-invariant probability measures $\mu$ on $X$, and $\pi\mu$ is the $S$-invariant probability measure on $Y$ induced via $\pi$.

More recently, M. Tsukamoto \cite{Tsu} introduced a new approach to defining weighted pressure and obtains the corresponding variational principle. In \cite{FH}, the authors posed several questions about extending the results of \eqref{orithm}, one of which concerns the existence of a relative version of \eqref{orithm}. Inspired by the ingenious ideas of M. Tsukamoto, we show that, for factor maps $\pi: (X, T) \to (Y, S)$ and $\varphi: (Y, S) \to (Z, R)$, and for $0 \leq \omega \leq 1$, we can define the relative weighted topological pressure $P^\omega_Z(\pi, T, f)$ for $(X, T)$ and $(Y, S)$ with respect to the common factor $(Z, R)$, and establish a relative weighted variational principle for it. Moreover, we generalize some results from the literature \cite{Lp, DS, YK} to a weighted version.
Note that the main proof mainly relies on the technique of zero-dimensional principal extensions, as developed in the work of T. Downarowicz and D. Huczek \cite{DH} (see also \cite{Dow}). Finally, we would like to mention the work of T. Wang and Y. Huang \cite{WY}, in which they discuss weighted entropy in relative settings and derive the relative Brin-Katok formula in a weighted context.

This paper is organized as follows. In Section 2, we provide the definitions of relative weighted topological pressure and establish some fundamental properties. At the end of this section, we state the main results of the paper. In Section 3, following the discussion in \cite{Tsu}, we prove some basic properties of relative weighted pressures. In Section 4, we recall the concept of zero-dimensional principal extensions and apply it to relative weighted pressures. In the final section, we prove the main theorems of the paper.

\section{Relative weighted topological pressure}
\subsection{Relative weighted topological pressures}
Let $X$ be a compact metric space and $T:X\to X$ a continuous self-map on $X$. We call the pair $(X,T)$ a  \textit{topological dynamical system} (TDS for short).
Consider a subset $\Omega\subset X$, a class $\mathcal{U}$ of subsets of $X$ is said to be a \textit{cover} of $\Omega$ if 
$\Omega\subset \bigcup_{U\in \mathcal{U}}U.$
We always assume that a cover is finite, and the class of finite cover (finite open cover, cover with disjoint subsets) of $\Omega$ is denoted by $\mathcal{C}_X(\Omega)$ (resp. $\mathcal{C}^o_X(\Omega)$, $\mathcal{P}_X(\Omega)$). Particularly, if $\Omega=X$, we simply write $\mathcal{C}_X$ (resp. $\mathcal{C}_X^o$, $\mathcal{P}_X$).

Let $\mathcal{U},\mathcal{V}\in\mathcal{C}_X$. $\mathcal{V}$ is said to be \textit{finer than} $\mathcal{U}$ (write $\mathcal{U}\preceq\mathcal{V}$) if for each $V\in \mathcal{V}$ there is $U\in \mathcal{U}$ such that $V\subset U$. As usual, we define 
$\mathcal{U}\vee\mathcal{V}=\{U\cap V:U\in\mathcal{U}, V\in\mathcal{V}\}.$
For any $n<m\in \mathbb{N}$, we define 
$\mathcal{U}^m_n=\bigvee_{i=n}^{m-1}T^{-i}\mathcal{U}$
and write $\mathcal{U}^n=\mathcal{U}_0^{n-1}$ for short.

Let $(X,T)$ be a TDS and $d$ a metric on $X$. For each $n\in \mathbb{N}$, we define a compatible metric on $X$ by
\begin{align}\label{eq1}
    d_n(x_1,x_2)=\max_{0\leq k<n}d(T^kx_1,T^kx_2)\text{ for all }x_1,x_2\in X.
\end{align}
In this paper, we use the symbol $\diam(U,d_n)$ to denote the diameter of   $U$ with respect to metric $d_n$.

We denote $C(X)$ the class of all real-valued continuous functions on $X$. For each $f\in C(X)$, we write 
$|f|=\sup_{x\in X}|f(x)|$ and define
\[\mathbb{S}_nf(x)=\sum_{k=0}^{n-1}f(T^kx)\text{ for all }x\in X.\]
To address the TDS $(X,T)$, sometimes, we write $d_n^T$ and $\mathbb{S}_n^Tf$ for specific.

Let $\Omega\subset X$. For each $n\in \mathbb{N}$, $\varepsilon>0$ and $f\in C(X)$, we define
\begin{align*}
P(T,\Omega,f,n,\varepsilon)=\inf_{\mathcal{U}\in \mathcal{C}^o_X(\Omega)}\left\{\sum_{U\in \mathcal{U}} e^{\sup_{x\in U}\mathbb{S}_nf(x)}: \diam(U,d_n)<\varepsilon\text{ for all }U\in \mathcal{U}\right\},
\end{align*}
then we define 
\[P(T,\Omega,f,\varepsilon)=\limsup_{n\to \infty}\frac{1}{n}\log P(T,\Omega,f,n,\varepsilon),\]
and the {topological pressure} of $\Omega$ is defined by 
$$P(T,\Omega,f)=\lim_{\varepsilon\to 0}P(T,\Omega,f,\varepsilon).$$
Particularly, if $\Omega=X$, the \textit{topological pressure} of $(X,T)$ is given by $$P(T,f)= P(T,X,f).$$ 
    If $f\equiv 0$, we define 
   $ h_{top}(T,\Omega,n,\varepsilon)=P(T,\Omega,0,n,\varepsilon),\,\,h_{top}(T,\Omega,\varepsilon)=P(T,\Omega,0,\varepsilon)$
    and the \textit{topological entropy} of $\Omega$ is defined by $$h_{top}(T,\Omega)=P(T,\Omega,0).$$

Then, we introduce the relative weighted pressure between TDSs. Our setting is based on two factor maps: $\pi: (X, T) \to (Y, S)$ and $\varphi: (Y, S) \to (Z, R)$ with the composite map $\psi = \varphi \circ \pi: (X, T) \to (Z, R)$.

Let $d$ and $d'$ be metrics on $X$ and $Y$, respectively, and $K\subset Z$. For each $\varepsilon>0$, $0\leq \omega\leq 1$, and $f\in C(X)$, we define 
\begin{align}\label{eq2.0}
P_Z^\omega&(\pi,T,K,f,n,\varepsilon)=\notag\\&\inf_{\mathcal{V}\in \mathcal{C}^o_Y(\varphi^{-1}K)}\left\{\sum_{V\in \mathcal{V}}(P(T,\pi^{-1}(V),f,n,\varepsilon))^\omega:\diam(V,d'_n)<\varepsilon\text{ for all }V\in \mathcal{V}\right\},
\end{align}
and set
\[
P_Z^\omega(\pi,T,K,f,\varepsilon)=\limsup_{n\to \infty}\frac{1}{n}\log P_Z^\omega(\pi,T,K,f,n,\varepsilon),
\]
then we define
\[
P_Z^\omega(\pi,T,K,f)=\lim_{\varepsilon\to 0}P_Z^\omega(\pi,T,K,f,\varepsilon).
\]
Particularly, if $K=\{z\}$ is a singleton, we write 
\[
P_Z^\omega(\pi,T,z,f,n,\varepsilon)=P_Z^\omega(\pi,T,\{z\},f,n,\varepsilon)
\quad \text{and} \quad
P_Z^\omega(\pi,T,z,f)=P_Z^\omega(\pi,T,\{z\},f).
\]
 
 In addition, if $f\equiv 0$, we put 
$h_Z^\omega(\pi,T,K)=P_Z^\omega(\pi,T,K,0).$

Here, we borrow a topological result in \cite[Chapter 3]{Ke} (see also \cite[Appendix A1]{DS}).

\begin{lemma}\label{lem2.7}
    Let $\pi: X\to Y$ be a quotient (surjective) map between two topological spaces. Then $\pi$ is closed if and only if for any open subset $U$ of $X$, the union of all fibers of $\pi$ contained in $U$ is open.
\end{lemma}

Recall that a real-valued function $f:X\to \mathbb{R}$ is \textit{upper semicontinuous} if the set $\{x\in X : f(x)<r\}$ is open for any $r\in \mathbb{R}$. 

\begin{proposition}\rm{(Upper semicontinuous)}\label{prop2.9}
      Let $f\in C(X)$ and $\varepsilon>0$. 
      \begin{enumerate}
          \item For any $n\in \mathbb{N}$, the function 
          $z\mapsto P_Z^\omega(\pi,T,z,f,n,\varepsilon)$
          is upper semicontinuous. Thus, 
          $z\mapsto P_Z^\omega(\pi,T,z,f)$
          is Borel measurable.
          \item There exists a constant $C(\varepsilon)>0$ such that 
              \[
              \frac{1}{n}\log P_Z^\omega(\pi,T,z,f,n,\varepsilon)<C(\varepsilon) \text{ for all }n\in \mathbb{N}, z\in Z.
              \]
             
      \end{enumerate}
\end{proposition}

\begin{proof}
    (1) Suppose $P_Z^\omega(\pi,T,z,f,n,\varepsilon)<C(z,n,\varepsilon)$ for some positive number $C(z,n,\varepsilon)$, then there is $\mathcal{V}=\{V_1,\cdots,V_p\}\in\mathcal{C}_Y^o(\varphi^{-1}(z))$ with $\diam(V_i,d'_n)<\varepsilon$, $i=1,\cdots,p,$ such that 
    \[\sum_{i=1}^p(P(T,\pi^{-1}(V_i),f,n,\varepsilon))^\omega<C(z,n,\varepsilon).\]
  By Lemma \ref{lem2.7} there is an open subset $W$ of $Y$ such that $W\subset V_1\cup\cdots\cup V_p$ consisting of fibers of $\varphi$ and $\varphi^{-1}(z)\subset W$. Then, we have
    \[\sum_{i=1}^p(P(T,\pi^{-1}(V_i\cap W),f,n,\varepsilon))^\omega<C(z,n,\varepsilon)\]
    Moreover, $\varphi(W)$ is also an open subset of $Z$\footnote{Every closed continuous surjective map is a quotient map, see \cite{Ke} Chapter 3.}. Hence, for any $z_0\in \varphi(W)$, $\varphi^{-1}(z_0)\subset (V_1\cap W)\cup\cdots\cup(V_p\cap W)$ and $P_Z^\omega(\pi,T,z_0,f,n,\varepsilon)<C(z,n,\varepsilon)$, which means $z\mapsto P_Z^\omega(\pi,T,z,f,n,\varepsilon)$ is upper semi-continuous and the function $z\mapsto P_Z^\omega(\pi,T,z,f)$ is Borel measurable.

    (2) Let $N(\varepsilon,Y)$ be the smallest number of open sets of diameter $\varepsilon$ required to cover $Y$ and $N(\varepsilon,X)$ the smallest number of open sets of diameter $\varepsilon$ required to cover $X$. Then  
    \begin{align*}
        P^\omega_Z(\pi,T,z,f,n,\varepsilon)\leq \sum_{j=1}^{N(\varepsilon,Y)^n}\left(\sum_{i=1}^{N(\varepsilon,X)^n}e^{n|f|}\right)^\omega 
    \end{align*}
  for all $n\in \mathbb{N}$ and $z\in Z$.
    Hence, by letting $C(\varepsilon)=\log N(\varepsilon,Y)+\omega\log N(\varepsilon,X)+|f|$ we obtain
    \begin{align*}
              \frac{1}{n}\log P_Z^\omega(\pi,T,z,f,n,\varepsilon)<C(\varepsilon).
          \end{align*}
          for all $n\in \mathbb{N}$ and $z\in Z.$
\end{proof}

Let $(X,T)$ be a TDS with metric $d$ on $X$. For any $m,n\in \mathbb{N}$ we define a pseudo-metric on $X$ by 
\[d_{m(+n)}(x_1,x_2)=\max_{0\leq k<m}d(T^{k+n}x_1,T^{k+n}x_2)\text{ for all }x_1,x_2\in X.\]
Note that $d_{m(+n)}$ is not necessarily a metric, but the ball $B_{d_{m(+n)}}(x,\varepsilon)=\{y:d_{m(+n)}(y,x)<\varepsilon\}=\bigcap_{0\leq k<m}T^{-(k+n)}B_d(T^{k+n}x,\varepsilon)$ is still open for each $x\in X$.
For each $f\in C(X)$, we define
\[\mathbb{S}_{m(+n)}f(x)=\sum_{k=n}^{n+m-1}f(T^kx).\]  

For any $\Omega\subset X$ and $\varepsilon>0$, we define
\begin{align*}
P(T,\Omega,&f,m(+n),\varepsilon)=
\\&\inf_{\mathcal{U}\in \mathcal{C}^o_X(\Omega)}\left\{\sum_{U\in \mathcal{U}} e^{\sup_{x\in U}\mathbb{S}_{m(+n)}f(x)}: \diam(U,d_{m(+n)})<\varepsilon\text{ for all }U\in \mathcal{U}\right\}.
\end{align*}
Similarly, for each $K\subset Z$ the quantity $P_Z^\omega(\pi,T,K,f,m(+n),\varepsilon)$ can be defined in the same way, that is, 
\begin{align*}
P_Z^\omega&(\pi,T,K,f,m(+n),\varepsilon)=\notag\\&\inf_{\mathcal{V}\in \mathcal{C}^o_Y(\varphi^{-1}K)}\left\{\sum_{V\in \mathcal{V}}(P(T,\pi^{-1}(V),f,m(+n),\varepsilon))^\omega:\diam(V,d'_{m(+n)})<\varepsilon\text{ for all }V\in \mathcal{V}\right\}.
\end{align*}

 Recall that a sequence $\mathcal{G}=\{g_n : n\in \mathbb{N}\}$ of nonnegative functions on TDS $(Z,R)$ is \textit{subadditive} if for any $m, n \in \mathbb{N}$ and $z \in Z$, we have
\[
g_{n+m}(z) \leq g_n(z) + g_m(R^n z).
\]
Then if $g_n\in \mathcal{G}$ are bounded for all $n \in \mathbb{N}$, it is clear that  
\[
\sup_{z \in Z} g_{m+n}(z) \leq \sup_{z \in Z} g_m(z) + \sup_{z \in Z} g_n(z).
\]
Thus, by Fekete's subadditive lemma 
\[
\lim_{n \to \infty} \frac{\sup_{z \in Z} g_n(z)}{n}=\inf_{n \in \mathbb{N}} \frac{\sup_{z \in Z} g_n(z)}{n}.
\]

Moreover, the well-known Kingman's subadditive theorem states that given $\kappa$ an $R$-invariant probability measure, and $\mathcal{G}=\{g_n : n \in \mathbb{N}\}$ a sequence of nonnegative integrable subadditive functions on $(Z,R)$. Then \[
    \lim_{n\to \infty}\frac{1}{n}g_n(z)\text{ exists }\kappa-a.e.,
    \text{ and }
    \int_Z\lim_{n\to \infty}\frac{1}{n}g_n(z)d\kappa(z)
    =\lim_{n\to \infty}\frac{1}{n}\int_Z g_n(z)d\kappa(z).
    \]
    In particular, if $\kappa$ is an $R$-invariant ergodic measure on $Z$, then 
    \[
    \lim_{n\to \infty}\frac{1}{n}g_n(z)
    =\lim_{n\to \infty}\frac{1}{n}\int_Z g_n(z)d\kappa(z)\quad \kappa-a.e.
    \]

\begin{proposition}\rm{(Subadditive)}\label{lemm212}
   Let $\pi:(X,T)\to (Y,S)$ and $\varphi:(Y,S)\to (Z,R)$. For each $\varepsilon>0$, $\{\log P_Z^\omega(\pi,T,\cdot,f,n,\varepsilon):n\in \mathbb{N}\}$ is a sequence of bounded nonnegative subadditive functions on $(Z,R)$.
\end{proposition}

\begin{proof}
Let $z\in Z$, $n\in \mathbb{N}$ and $\mathcal{V}=\{V_1,\cdots,V_p\}\in \mathcal{C}^o_Y(\varphi^{-1}(z))$ with $\diam(V_i,d'_n)< \varepsilon$ for all $i=1,\cdots,p$. Take $\mathcal{U}=\{U_1,\cdots, U_q\}\in \mathcal{C}^o_Y(\varphi^{-1}(z))$  with $\diam(U_j,d'_{m(+n)})<\varepsilon$ for all $j=1,\cdots,q$. Then
         \begin{align}\label{eq(2.22)}
             &\left(\sum_{i=1}^p(P(\pi,T,\pi^{-1}(V_i),f,n,\varepsilon))^\omega\right)\cdot\left(\sum_{j=1}^q(P(\pi,T,\pi^{-1}(U_j),f,m(+n),\varepsilon))^\omega\right)
             \notag\notag\\ &\geq \sum_{i=1}^p\sum_{j=1}^q\left(P(\pi,T,\pi^{-1}(V_i),f,n,\varepsilon)\cdot P(\pi,T, \pi^{-1}(U_j),f,m(+n),\varepsilon)\right)^\omega\notag\\
             &\geq \sum_{i,j}(P(T,\pi^{-1}(V_i\cap U_j),f,m+n,\varepsilon))^\omega.
             \end{align}
             It is easy to check that $\mathcal{V}\vee\mathcal{U}$ is a class of open subsets that covers $\varphi^{-1}(z)$ with $\diam(U\cap V,d'_{m+n})< \varepsilon$ for any $U\cap V\in \mathcal{V}\vee\mathcal{U}$, hence,
             \begin{align*}
                \eqref{eq(2.22)} \geq P_Z^\omega(\pi,T,z,f,m+n,\varepsilon).
             \end{align*}
         As $\mathcal{V}$ and $\mathcal{U}$ can be taken arbitrarily, we obtain 
         \[ P_Z^\omega(\pi,T,z,f,m+n,\varepsilon)\leq  P_Z^\omega(\pi,T,z,f,n,\varepsilon)\cdot P_Z^\omega(\pi,T,z,f,m(+n),\varepsilon),\]

    Let $\mathcal{V}=\{V_1,\cdots,V_p\}\in \mathcal{C}^o_Y(\varphi^{-1}(R^nz))$ with $\diam(V_i,d'_m)<\varepsilon$, $i=1,\cdots,p$,   and $\mathcal{U}=\{U_{ij}:1\leq i\leq p, 1\leq j\leq \beta_i\}$ a class of open subsets in $X$ with $\diam(U_{ij},d_m)<\varepsilon$ and $\pi^{-1}(V_i)\subset \bigcup_{1\leq j\leq \beta_i }U_{ij}$ such that
      \begin{align}\label{qqq2.51}
\sum_{i=1}^p\left(\sum_{U_{ij}}e^{\sup_{U_{ij}}\mathbb{S}_mf}\right)^\omega\leq P^\omega_Z(\pi,T,R^nz,f,m,\varepsilon)+\delta.
      \end{align}
     Then $\varphi^{-1}(z)\subset S^{-n}\varphi^{-1}(R^nz)\subset S^{-n}V_1\cup\cdots\cup S^{-n}V_p$ with $\diam(S^{-n}V_i,d'_{m(+n)})<\varepsilon$ and $T^{-n}\mathcal{U}$ is a class of open subsets in $X$ such that $\pi^{-1}(S^{-n}V_i)\subset \bigcup_{1\leq j\leq \beta_i}T^{-n}U_{ij}$ and $\diam(T^{-n}U_{ij},d_{m+(n)})<\varepsilon$ and we have
       \begin{align*}
           \sum_{S^{-n}V_j}\left(\sum_{T^{-n}U_{ij}}e^{\sup_{T^{-n}U_{ij}}\mathbb{S}_{m(+n)}f}\right)^\omega=\sum_{i=1}^p\left(\sum_{U_{ij}}e^{\sup_{U_{ij}}\mathbb{S}_mf}\right)^\omega.
       \end{align*}
       Thus,
       \[ P_Z^\omega(\pi,T,z,f,m(+n),\varepsilon)\leq\sum_{S^{-n}V_j}\left(\sum_{T^{-n}U_{ij}}e^{\sup_{T^{-n}U_{ij}}\mathbb{S}_{m(+n)}f}\right)^\omega.\]
       Therefore, combining \eqref{qqq2.51} and as $\delta$ can be arbitrarily chosen, we have
       \[ P_Z^\omega(\pi,T,z,f,m+n,\varepsilon)\leq  P_Z^\omega(\pi,T,z,f,n,\varepsilon)\cdot P^\omega_Z(\pi,T,R^nz,f,m,\varepsilon),\]
       which means that $\{\log P_Z^\omega(\pi,T,\cdot,f,n,\varepsilon):n\in \mathbb{N}\}$ is a sequence of bounded nonnegative subadditive functions on $Z$.
\end{proof}
Combining Kingman's subadditive theorem and Proposition \ref{lemm212}, we have the following statement.
\begin{theorem}\label{thm2.5}
 For any $\varepsilon>0$ and $f\in C(X)$, we have
 \begin{enumerate}
     \item [\rm{(1)}]  the limit \[P^\omega_Z(\pi,T,f,\varepsilon)=\lim_{n\to \infty}\frac{\sup_{z\in Z}\log P^\omega_Z(\pi,T,z,f,n,\varepsilon)}{n}\] exists;
     \item [\rm{(2)}] if $\kappa$ is an $R$-invariant measure on $Z$, then 
  \[P^\omega_Z(\pi,T,z,f,\varepsilon)=\lim_{n\to \infty}\frac{1}{n}\log P^\omega_Z(\pi,T,z,f,n,\varepsilon)\quad \kappa-a.e.,\]
   and
    \begin{align*}
        \int_ZP^\omega_Z(\pi,T,z,f,\varepsilon)d\kappa(z)=\lim_{n\to \infty}\frac{1}{n}\int_Z\log P^\omega_Z(\pi,T,z,f,n,\varepsilon)d\kappa(z).
    \end{align*}
 \end{enumerate}
 
\end{theorem}

\begin{definition}\label{defn2.1}
   Let $\pi:(X,T)\to (Y,S)$ and $\varphi:(Y,S)\to (Z,R)$ be two factor maps. For each $f\in C(X)$ and $0\leq\omega\leq 1$, we define
   \begin{align}\label{eqq2.3}
P^\omega_Z(\pi,T,f)=\lim_{\varepsilon\to 0}\left(\lim_{n\to \infty}\frac{\sup_{z\in Z}\log P^\omega_Z(\pi,T,z,f,n,\varepsilon)}{n}\right).
\end{align}
to be the $\omega$-\textit{relative weighted topological pressure} of $\pi$.

If $f\equiv 0$, we define the $\omega$-\textit{relative weighted topological entropy} of $\pi$ by
    \[h^\omega_Z(\pi,T)=\lim_{\varepsilon\to 0}\left(\lim_{n\to \infty}\frac{\sup_{z\in Z}\log P^\omega_Z(\pi,T,z,0,n,\varepsilon)}{n}\right),\]
\end{definition}

\begin{remark}
   If $Z=\{*\}$ a singleton, the definition \eqref{eqq2.3} returns to the $\omega$-\textit{weighted topological pressure} for the factor map $\pi:(X,T)\to (Y,S)$ defined in \cite{Tsu}, that is,
    \[P^\omega(\pi,T,f)=\lim_{\varepsilon\to 0}\left(\lim_{n\to \infty}\displaystyle\frac{\log P_{*}^\omega(\pi, T, *,f,n,\varepsilon)}{n}\right).\] 
\end{remark}

\subsection{Conditional metric entropy}

Let $(X,T)$ be a TDS. We denote $\mathcal{M}(X)$, $\mathcal{M}(X, T)$  by the set of Borel probability measure, $T$-invariant probability measure, respectively. Given $\mu\in \mathcal{M}(X)$, consider the probability measure space $(X,\mathcal{B}_X,\mu)$ and $\mathcal{A}\in \mathcal{P}_X$. The \textit{partition entropy} of $\mathcal{A}$ is defined by 
\[H_\mu(\mathcal{A})=\sum_{A\in \mathcal{A}}-\mu(A)\log\mu(A),\]
 we assume $0\log0=0$. If $A$ is a subset of $X$ with $\mu(A)>0,$ write $\mu_A(B)=\mu(A\cap B)/\mu(A)$ for all $B\in \mathcal{B}_X$.  Let $\mathcal{B}\in \mathcal{P}_X$ be another finite partition of $X$, \textit{the conditional entropy of} $\mathcal{B}$ with respect to $\mathcal{A}$ is defined by 
\[H_\mu(\mathcal{B}|\mathcal{A})=\sum_{A\in \mathcal{A}}\mu(A)H_{\mu_A}(\mathcal{B})=\int_XH_{\mu_{A}}(\mathcal{B})d\mu(x).\]

Let $\pi:(X,T)\to (Y,S)$ be a factor map and $\mu\in \mathcal{M}(X,T)$, we write $\nu=\pi\mu(A)=\mu(\pi^{-1}A)$ for all $A\in \mathcal{B}_Y$, then $\nu\in \mathcal{M}(Y,S)$. Recall that $\mu$ admits a disintegration $\mu=\int_Y\mu_yd\nu(y)$ over $Y$, where $\mu_y$ is the fiber measure ($\mu_y(\pi^{-1}(y))=1$), and for each $\mathcal{A}\in \mathcal{P}_X$,  we define
\[H_\mu(\mathcal{A}|Y)=H_\mu(\mathcal{A}|\pi^{-1}\mathcal{B}_Y)=\int_YH_{\mu_y}(\mathcal{A})d\nu(y),\]
then the \textit{relative entropy} of $\mathcal{A}$ with respect to $\pi$ is defined by 
\[h_\mu(T,\mathcal{A}|Y)=\lim_{n\to \infty}\frac{1}{n}H_\mu(\mathcal{A}^{n-1}_0|\pi^{-1}\mathcal{B}_Y)=\lim_{n\to \infty}\frac{1}{n}\int_Y H_{\mu_y}(\mathcal{A}_0^{n-1})d\nu(y),\]
and the \textit{relative entropy} $h_\mu(T|S)$ of $(X,T)$ with respect to $(Y,S)$ is defined as follows (see \cite{Lp})
\[h_\mu(T|S)=\sup\{h_\mu(T,\mathcal{A}|Y)|\mathcal{A}\in \mathcal{P}_X\}.\]

We have the following standard properties (cf. \cite{Pw}).

\begin{lemma}\label{lem2.33}
    Let $\pi: (X, T) \to (Y, S)$ be a factor map, and let $\mu \in \mathcal{M}(X, T)$. For any $\mathcal{A}, \mathcal{B} \in \mathcal{P}_X$, the following hold:
    \begin{enumerate}
        \item $H_\mu(\mathcal{A} \vee \mathcal{B} \mid Y) \leq H_\mu(\mathcal{A} \mid Y) + H_\mu(\mathcal{B} \mid Y).$
        \item $h_\mu(T, \mathcal{A} \mid Y) \leq h_\mu(T, \mathcal{B} \mid Y) + H_\mu(\mathcal{A} \mid \mathcal{B}).$
    \end{enumerate}
\end{lemma}

\subsection{Main results}

With the above notations, recall that Leddrapier and Walters in \cite{Lp} prove the following result:
\begin{theorem}\label{thmm2.3}
   Let $\pi:(X,T)\to (Y,S)$ be a factor map. For any $f\in C(X)$ and $\nu\in \mathcal{M}(Y,S)$, 
    \[\int_YP(T,\pi^{-1}(y),f)d\nu(y)=\sup\left(h_\mu(T|S)+\int_Xfd\mu\right),\]
    where the supremum is taken over all $\mu\in \mathcal{M}(X,T)$ with $\nu=\pi\mu$.
\end{theorem}

Consider factor maps $\pi:(X,T)\to (Y,S)$, $\varphi:(Y,S)\to (Z,R)$ and $0\leq \omega\leq 1$. We can now state the weighted version of Leddrapier-Walter's type variational principle.
\begin{theorem}\rm{(Variational Principle I)}\label{thm2.6}
   Let $\pi:(X,T)\to (Y,S)$ and $\varphi:(Y,S)\to (Z,R)$ with $\psi=\varphi\circ\pi$ and $0\leq \omega\leq 1$. For any $f\in C(X)$ and $\kappa\in \mathcal{M}(Z,R)$, we have 
    \begin{align*}
        \int_ZP_Z^\omega(\pi,T,z,f)d\kappa(z)=\sup\left(\omega h_\mu(T|R)+(1-\omega)h_{\pi\mu}(S|R)+\omega\int_Xfd\mu\right),
    \end{align*}
    where the supremum is taken over all $\mu\in \mathcal{M}(X,T)$ with $\kappa=\psi\mu=\varphi\circ\pi(\mu)$.
\end{theorem}
Let $(X,T)$ be a TDS and $K\subset X$. We define \[N(T,K,n,\varepsilon)=\inf_{\mathcal{U}\in \mathcal{C}^o_X(K)}\left\{|\mathcal{U}|:\diam(U,d_n)<\varepsilon\text{ for all }U\in \mathcal{U}\right\}.\]
Let $\pi:(X,T)\to (Y,S)$ be a factor map. Recall the \textit{topological conditional entropy} $h_{top}(T,X|Y)$ of $\pi$ is defined by 
\begin{align}\label{eqqq2.3}
    h_{top}(T,X|Y)=\lim_{\varepsilon\to 0}\left(\lim_{n\to \infty}\frac{\sup_{y\in Y}\log N(T,\pi^{-1}y,n,\varepsilon)}{n}\right).
\end{align}

In \cite{DS}, Downarowicz and Serafin introduced the notions of relative topological entropy $h_{top}(T, X|Y)$. With the relative measure-theoretical entropy $h_\mu(T|S)$ for invariant measure $\mu\in\mathcal{M}(X,T)$, they proved the following relative variational principle:
  \begin{theorem}\label{thmm2.5}
      Let $\pi:(X,T)\to (Y,S)$ be a factor map. Then 
      \[h_{top}(T,X|Y)=\sup_{y\in Y}h_{top}(T,\pi^{-1}y)=\sup_{\mu\in \mathcal{M}(X,T)}h_\mu(T|S).\]
  \end{theorem}

Let $\pi:(X,T)\to (Y,S)$ and $\varphi:(Y,S)\to (Z,R)$ be two factor maps, $f\in C(X)$ and $0\leq \omega\leq 1$. We state the variational principles for $\omega$-relative weighted topological pressure as follows:
\begin{theorem}\rm{(Variational principle II)}\label{thm2.7}
   For any $f\in C(X)$ and $0\leq \omega\leq 1$, we have
      \[P^\omega_Z(\pi,T,f)=\sup_{\mu\in \mathcal{M}(X,T)}\left(\omega h_\mu(T|R)+(1-\omega)h_{\pi\mu}(S|R)+\omega\int_Xfd\mu\right)\] and
      \[\sup_{z\in Z}P_Z^\omega(\pi,T,z,f)=\sup_{\mu\in \mathcal{M}(X,T)}\left(\omega h_\mu(T|R)+(1-\omega)h_{\pi\mu}(S|R)+\omega\int_Xfd\mu\right).\]
      Therefore,
  \[ \sup_{z\in Z}P_Z^\omega(\pi,T,z,f)=P^\omega_Z(\pi,T,f).\]
\end{theorem}

By taking $f\equiv 0$, we obtain variational principles for entropy, that is,
\begin{corollary}
      Let $\pi:(X,T)\to (Y,S)$ and $\varphi:(Y,S)\to (Z,R)$ with $\psi=\varphi\circ\pi$ and $\kappa\in \mathcal{M}(Z,R)$. Given $0\leq \omega\leq 1$. \begin{enumerate}
          \item [\rm{(1)}] From Theorem \ref{thm2.6} we have  \begin{align*}
        \int_Zh_Z^\omega(\pi,T,z)d\kappa(z)=\sup\left(\omega h_\mu(T|R)+(1-\omega)h_{\pi\mu}(S|R)\right),
    \end{align*}
    where the supremum is taken over all $\mu\in \mathcal{M}(X,T)$ with $\kappa=\psi\mu$.
    \item [\rm{(2)}] From Theorem \ref{thm2.7} we have
    \begin{align*}
         h^\omega_Z(\pi,T)=\sup_{\mu\in \mathcal{M}(X,T)}\left(\omega h_\mu(T|R)+(1-\omega)h_{\pi\mu}(S|R)\right)=\sup_{z\in Z}h_Z^\omega(\pi,T,z).
    \end{align*}
   
      \end{enumerate}
\end{corollary}

\section{Basic properties} In this section, we prove some useful properties.
In \cite{Tsu}, Tsukamoto has established several fundamental properties for $\omega$-weighted topological pressure. We find that the proofs of the relative version are similar, but for completeness, we prove some of them.

\begin{proposition}\label{prop2.2}
    Let $\pi:(X,T)\to (Y,S)$ and $\varphi:(Y,S)\to (Z,R)$ with $\psi=\varphi\circ\pi$. For each $k\in \mathbb{N}$, we have
    \[P^\omega_Z(\pi,T^k,\mathbb{S}_k^Tf)=kP^\omega_Z(\pi,T,f),\]
     and for any $\kappa\in \mathcal{M}(Z,R)$,
     \[\int_ZP^\omega_Z(\pi,T^k,z,\mathbb{S}_k^Tf)d\kappa(z)=k\int_ZP^\omega_Z(\pi,T,z,f)d\kappa(z).\]
\end{proposition}
\begin{proof}
Let $d,d'$ be metrics on $X,$ $Y$, respectively. For any $\varepsilon>0$, there is $0<\delta<\varepsilon$ such that 
\[d^T(x_1,x_2)<\delta\Longrightarrow d_k^T(x_1,x_2)<\varepsilon,\text{ for all }x_1,x_2\in X,\]
\[{d'}^{S}(y_1,y_2)<\delta\Longrightarrow {d'}_k^{S}(y_1,y_2)<\varepsilon, \text{ for all }y_1,y_2\in Y.\]
Then for any $n\in \mathbb{N}$,
\[d^{T^k}_n(x_1,x_2)<\delta\Longrightarrow d^T_{kn}(x_1,x_2)<\varepsilon,\text{ for all }x_1,x_2\in X,\]
\[{d'}^{S^k}_n(y_1,y_2)<\delta\Longrightarrow {d'}_{kn}^S(y_1,y_2)<\varepsilon,\text{ for all }y_1,y_2\in Y.\]

    Let $z\in Z$ and $\mathcal{V}=\{V_1,\cdots, V_p\}\in \mathcal{C}^o_Y(\varphi^{-1}(z))$ with $\diam(V_i,{d'}^{S^k}_n)<\delta$ then $\diam(V_i,{d'}_{kn}^S)<\varepsilon$ for all $i=1,\cdots,p$, and if $\mathcal{U}_i=\{U_1,\cdots,U_{\beta_i}\}\in \mathcal{C}^o_X(\pi^{-1}(V_i))$ with $\diam(U_j,d^{T^k}_n)<\delta$ then $\diam(U_j,d^T_{kn})<\varepsilon$ for each $j=1,\cdots,\beta_i$. Hence,
    \[P^\omega_Z(\pi,T,z,f,kn,\varepsilon)\leq P^\omega_Z(\pi,T^k,z,\mathbb{S}_k^Tf,n,\delta).\]
    Because $\mathbb{S}_n^{T^k}(\mathbb{S}_k^T)=\mathbb{S}_{kn}^T$ and $d_{kn}^T(x_1,x_2)<\varepsilon$ (resp. ${d'}^S_{kn}(y_1,y_2)<\varepsilon$) implies $d^{T^k}_n(x_1,x_2)<\varepsilon$ (resp. ${d'}^{S^k}_n(y_1,y_2)<\varepsilon$), we have
    \[P^\omega_Z(\pi,T^k,z,\mathbb{S}^T_kf,n,\varepsilon)\leq P^\omega_Z(\pi,T,z,f,kn,\varepsilon).\]
    Thus,
    \[P^\omega_Z(\pi,T^k,z,\mathbb{S}^T_kf,n,\varepsilon)\leq P^\omega_Z(\pi,T,z,f,kn,\varepsilon)\leq P^\omega_Z(\pi,T^k,z,\mathbb{S}_k^Tf,n,\delta).\]
    Therefore, 
    \[\sup_{z\in Z}P^\omega_Z(\pi,T^k,z,\mathbb{S}^T_kf,n,\varepsilon)\leq \sup_{z\in Z} P^\omega_Z(\pi,T,z,f,kn,\varepsilon)\leq \sup_{z\in Z}P^\omega_Z(\pi,T^k,z,\mathbb{S}_k^Tf,n,\delta),\]
    and
    \[P^\omega_Z(\pi,T^k,\mathbb{S}^{T}_k,f)=kP^\omega_Z(\pi,T,f).\]

    Let $\kappa\in\mathcal{M}(Z,R)$. By Theorem \ref{thm2.5} we have
    \[\int_ZP^\omega_Z(\pi,T^k,z,\mathbb{S}^T_kf,\varepsilon)d\kappa(z)\leq k\int_ZP^\omega_Z(\pi,T,z,f,\varepsilon)d\kappa(z)\leq \int_ZP^\omega_Z(\pi,T^k,z,\mathbb{S}_k^Tf,\delta)d\kappa(z).\]
    Let $\varepsilon$ and $\delta$ approach to $0$, we have
    \[\int_ZP^\omega_Z(\pi,T^k,z,\mathbb{S}_k^Tf)d\kappa(z)=k\int_ZP^\omega_Z(\pi,T,z,f)d\kappa(z).\]
\end{proof}

The relative weighted topological pressure possesses the following property. For non-relative case, one can see \cite[Lemma 2.3]{Tsu} for details.

\begin{proposition}\label{prop2.3}
   Suppose $(X_i,T_i)$ $(i=1,2,3)$ and $(Z,R)$ are TDSs admitting the following commutative diagram:
    \begin{align}\label{dia2.3}
        \xymatrix@C=60pt{
(X_3,T_3) \ar[r]^-{\pi_2} \ar[rd]_-{\psi_2} & (X_2,T_2) \ar[r]^-{\pi_1} \ar[d]^-{\psi_1} & (X_1,T_1) \ar[ld]^-{\varphi} \\
 & (Z,R) & } 
    \end{align} 
    Then for each $f\in C(X_2)$ and $z\in Z$,
    \[P^\omega_Z(\pi_1,T_2,z,f)\leq P^\omega_Z(\pi_1\circ\pi_2,T_3,z,f\circ\pi_2),\]and
    \[P^\omega_Z(\pi_1,T_2,f)\leq P^\omega_Z(\pi_1\circ\pi_2,T_3,f\circ\pi_2).\]
\end{proposition}
\begin{proof}
    Let $d^i$ be metrics on $X_i$, $i=1,2,3$ and $\varepsilon>0$. For each $n\in\mathbb{N}$, $d^i_n$ is defined as in \eqref{eq1}. There is a $0<\delta<\varepsilon$ such that $d^3(x_1^3,x_2^3)<\delta$ implies that $d^2(\pi_2(x_1^3),\pi_2(x_2^3))<\varepsilon$ for all $x^3_1,x^3_2\in X_3.$ Then for any $n>0$,
    \[d_n^3(x_1^3,x_2^3)<\delta\Longrightarrow d_n^2(\pi_2(x_1^3),\pi_2(x_2^3))<\varepsilon.\]
    Hence, for any $\Omega\subset X_3$, we have
    \[P(\pi_2(\Omega),T_2,f,n,\varepsilon)\leq P(\Omega,T_3,f\circ \pi_2,n,\delta).\]
    Let $\mathcal{V}_1=\{V_1,\cdots,V_p\}\in \mathcal{C}^o_Y(\varphi^{-1}(z))$ with $\diam(V_i,d^1_n)<\delta$. Then for any $i=1,\cdots,p$, we have
    \[P(\pi_1^{-1}(V_i),T_2,f,n,\varepsilon)\leq P(\pi_2^{-1}(\pi_1^{-1}(V_i)),T_3,f\circ \pi_2,n,\delta).\]
    Therefore,
    \[P_Z^\omega(\pi_1,T_2,z,f,n,\varepsilon)\leq P_Z^\omega(\pi_1\circ\pi_2,T_3,z,f\circ\pi_2,n,\delta).\]
    Thus,
    \[P^\omega_Z(\pi_1,T_2,z,f)\leq P^\omega_Z(\pi_1\circ\pi_2,T_3,z,f\circ\pi_2),\] and
    \[P^\omega_Z(\pi_1,T_2,f)\leq P^\omega_Z(\pi_1\circ\pi_2,T_3,f\circ\pi_2),\] which completes the proof.
\end{proof}

The following property is a relative version of Lemma 2.4 in \cite{Tsu}. The proof is nearly the same as in \cite{Tsu}, so we state it without proof and one can see more details in \cite{Tsu}
\begin{proposition}\label{prop2.4}
  Assume that the following solid line commutative diagram exists among the dynamical systems $(X,T)$, $(Y,S)$, $(Y',S')$ and $(Z,R)$:
  \begin{align}\label{dia2.4}
      \xymatrix@C=60pt{
(X,T) \ar[rr]^{\pi} \ar[rd] &  & (Y,S) \ar[ld]^\psi \\
 & (Z,R) & \\
 (X',T')\ar@{-->}[uu]^{\eta} \ar@{-->}[ur] \ar@{-->}[rr]_{\Pi}& & (Y',S')\ar[ul]^{\phi} \ar[uu]_{\xi} }
  \end{align} 
  Then there is a dynamical system $(X',T')$ satisfying the commutative diagram as above such that for each $f\in C(X)$ and $z\in Z$, we have
   \[P^\omega_Z(\pi,T,z,f)\leq P^\omega_Z(\Pi,T',z,f\circ\eta)\text{ and }P^\omega_Z(\pi,T,f)\leq P^\omega_Z(\Pi,T',f\circ\eta).\]
\end{proposition}

\section{Zero-dimensional principal extension revisited}

Recall that a factor map $\pi:(X,T)\to (Y,S)$ is said to be \textit{principal} if $h_{top}(T,X|Y)=0$, where $h_{top}(T,X|Y)$ is the conditional topological entropy of $(X,T)$ with respect to $(Y,S)$ defined as \eqref{eqqq2.3}. 
We need the following significant result for principal extension, which is contained in \cite{Dow}.
\begin{theorem}\rm{(\cite[Corollary 6.8.9]{Dow})}\label{thm4.1}
    Let $\pi:(X,T)\to (Y,S)$ be a factor map with $h_{top}(Y,S)<\infty$, $\pi$ is principal if and only if for any $\mu\in \mathcal{M}(X,T)$, we have
    \[h_\mu(T)=h_{\pi\mu}(S).\]
\end{theorem}

The following property is proved in \cite{Tsu}, we restate it here for a relative version.
\begin{lemma}\rm{(\cite[Lemma 5.3 with $Z=\{*\}$]{Tsu})}\label{lem4.8}
    Suppose the commutative diagram \eqref{dia2.4} holds as in Proposition \ref{prop2.4}
 and
  \[X'=X\times_YY'=\{(x,y)\in X\times Y'|\pi(x)=\xi(y)\}.\]
  If $\xi$ is a principal extension between $(Y',S')$ and $(Y,S)$, then $\eta$ is also a principal extension between $(X,T)$ and $(X\times_YY',T\times S')$.
\end{lemma}

Recall that a compact metric space $X$ is said to be \textit{zero-dimensional} if it has a base consisting of clopen sets. For a topological dynamical system $(X,T)$, the following significant result is proved in \cite[Theorem 3.1]{DH} and contained in \cite[Theorem 7.6.1]{Dow}.
\begin{theorem}\label{thm4.9}
    Let $(X,T)$ be a TDS, there is an extension map $\pi:(X',T')\to(X,T)$ such that 
    \begin{enumerate}
        \item  $\pi:(X',T')\to (X,T)$ is principal;
        \item $X'$ is a zero-dimensional compact metrizable space.
    \end{enumerate}
\end{theorem}

The following theorem, known as the Rohlin-Abramov theorem (see e.g.,  \cite[Lemma 3.1]{Lp}), plays an important role in the proof of Proposition \ref{prop4.10}.
\begin{theorem}\label{thm3.2}
    Let $\pi:(X,T)\to (Y,S)$ and $\varphi:(Y,S)\to (Z,R)$ be two factor maps and $\mu\in\mathcal{M}(X,T)$, then 
    \[h_\mu(T|R)=h_{\mu}(T|S)+h_{\pi\mu}(S|R).\]
\end{theorem}

\begin{remark}
    Let $\pi,\varphi$ be factor maps as above. If $\pi$ is a principal extension between $(X,T)$ and $(Y,S)$, then by Theorem \ref{thm4.1},
    \begin{align*}
        h_\mu(T|R)&=h_{\mu}(T|S)+h_{\pi\mu}(S|R)\\&=h_{\mu}(T)-h_{\pi\mu}(S)+h_{\pi\mu}(S|R)=h_{\pi\mu}(S|R).
    \end{align*}
    
\end{remark}

 We now state a key property for the relative weighted topological pressure as follows. For convenience, we first put \[P^\omega_{Z,var}(\pi,T,f)=\sup_{\mu\in \mathcal{M}(X,T)}\left(\omega h_\mu(T|R)+(1-\omega)h_{\pi\mu}(S|R)+\omega\int_Xfd\mu\right).\]

\begin{proposition}\rm{(\cite[Corollary 5.5 with $Z=\{*\}$]{Tsu})}\label{prop4.10}
    Let $\pi:(X,T)\to (Y,S)$ and $\varphi:(Y,S)\to (Z,R)$ with $\psi=\varphi\circ\pi$ and $f\in C(X)$. There is a commutative diagram satisfying
    \begin{align}
      \xymatrix@C=60pt{
(X,T) \ar[rr]^{\pi} \ar[rd]_\psi &  & (Y,S) \ar[ld]^\varphi \\
 & (Z,R) & \\
 (X\times_YY',T\times S')\ar[uu]^{\eta} \ar@{-->}[ur] \ar[rr]_{\Pi}& & (Y',S')\ar@{-->}[ul]^{\phi} \ar[uu]_{\xi} \\ (X',T')\ar[u]^\rho \ar[urr]_{\Pi'}}
  \end{align} 
  \begin{enumerate}
      \item  The factor maps $\eta$, $\rho$ (hence $\eta\circ \rho$) and $\xi$ are principal extensions. Besides, $X'$ and $Y'$ are zero-dimensional.
      \item  For any $0\leq \omega\leq 1$ and $z\in Z$, we have
      \[P^\omega_Z(\pi,T,z,f)\leq P^\omega_Z(\Pi',T',z,f\circ\eta\circ\rho)\] and \[P^\omega_Z(\pi,T,f)\leq P^\omega_Z(\Pi',T',f\circ\eta\circ\rho).\] Moreover,
      \[P^\omega_{Z,var}(\Pi',T',f\circ\eta\circ\rho)\leq P^\omega_{Z,var}(\pi,T,f).\]
  \end{enumerate}
\end{proposition}
\begin{proof}
    By Theorem \ref{thm4.9} there is a zero-dimensional principal extension $\xi:(Y',S')\to (Y,S)$. Let $(X\times_Y Y',T\times S')$ be the joining of $(X,T)$ and $(Y',S')$ over $(Y,S)$ and let $\eta:X\times_YY'\to X$ and $\Pi:X\times_YY'\to Y'$ be the projections. By Proposition \ref{prop2.4}, for any $z\in Z$ it holds that
    \[P^\omega_Z(\pi,T,z,f)\leq P^\omega_Z(\Pi,T\times S',z,f\circ\eta),\]
    and 
    \[P^\omega_Z(\pi,T,f)\leq P^\omega_Z(\Pi,T\times S',f\circ\eta).\]

    From the commutative diagram, for any $\mu\in \mathcal{M}(X\times_YY',T\times S')$ it holds that 
    \begin{align}\label{2.9}
        h_{\pi\eta\mu}(S|R)=h_{\xi\Pi\mu}(S|R),
    \end{align}
and by Rohlin-Abramov Theorem, we have 
\begin{align}\label{2.10}
    h_\mu(T\times S'|R)=h_\mu(T\times S'|T)+h_{\eta\mu}(T|R)
\end{align}
and
\begin{align}\label{2.11}
    h_{\Pi\mu}(S'|R)=h_{\Pi\mu}(S'|S)+h_{\xi\Pi\mu}(S|R)
\end{align}
    
  By Lemma \ref{lem4.8} $\eta$ is principal and as $\xi$ is principal, \eqref{2.9}, \eqref{2.10} and \eqref{2.11} imply that 
    \begin{align}\label{eqqq4.6}
        h_\mu(T\times S'|R)=h_{\eta\mu}(T|R) \text{ and }h_{\Pi\mu}(S'|R)=h_{\xi\Pi\mu}(S|R)=h_{\pi\eta\mu}(S|R). 
    \end{align}
    Therefore, from \eqref{eqqq4.6},
    \begin{align*}
        P^\omega_{Z,var}&(\Pi,T\times S',f\circ\eta)
        \leq P^\omega_{Z,var}(\pi,T,f).
    \end{align*}
    Using Theorem \ref{thm4.9} again, there is a zero-dimensional principal extension $\rho:(X',T')\to (X\times_YY',T\times S')$ as above. 
    By Proposition \ref{prop2.3}, we obtain
    \[P^\omega_Z(\Pi,T\times S',z,f\circ\eta)\leq P^\omega_Z(\Pi',T',z,f\circ\eta\circ\rho)\] and
    \[P_Z^\omega(\Pi,T\times S',f\circ \eta)\leq P_Z^\omega(\Pi',T',f\circ \eta\circ\rho).\]
    Using Rohlin-Abramov Theorem again, since $\rho$ is principal, we have
    \[P^\omega_{Z,var}(\Pi',T',f\circ\eta\circ\rho)\leq P^\omega_{Z,var}(\Pi,T\times S',f\circ\eta).\]
   Hence, 
    \[P^\omega_Z(\pi,T,z,f)\leq P_Z^\omega(\Pi',T',z,f\circ \eta\circ\rho),\]
    and
    \[P^\omega_Z(\pi,T,f)\leq P_Z^\omega(\Pi',T',f\circ \eta\circ\rho).\]
    Moreover,
    \[P^\omega_{Z,var}(\Pi',T',f\circ\eta\circ\rho)\leq P^\omega_{Z,var}(\pi,T,f).\]
\end{proof}
\section{Variational principles}
\subsection{Proof of one side of variational principles}
In this subsection, we prove that the weighted topological pressure is larger than the weighted measure-theoretic one.
\begin{lemma}\rm{(\cite[Lemma 9.9]{Pw})}\label{lem3.3}
    Let $c_i\in \mathbb{R}$ and $p_i\geq0$, $i=1,\cdots,m,$ with $\sum_{i=1}^mp_i=1$. Then we have
    \[\sum_{i=1}^mp_i(c_i-\log p_i)\leq\log\sum_{i=1}^me^{c_i}.\]
\end{lemma}

\begin{proposition}  
\label{prop4.1}
 Let $\pi:(X,T)\to (Y,S)$ and $\varphi:(Y,S)\to (Z,R)$ with $\psi=\varphi\circ\pi$ and $f\in C(X)$.  For any $0\leq \omega\leq 1$ and $\mu\in \mathcal{M}(X,T),$ the half of the variational principles hold:
    \begin{enumerate}
        \item [\rm{(1)}] We have
        \[\omega h_\mu(T|R)+(1-\omega)h_{\pi\mu}(S|R)+\omega\int_Xfd\mu\leq P^\omega_{Z}(\pi,T,f).\]
    \item [\rm{(2)}]
    If $\kappa=\psi\mu\in\mathcal{M}(Z,R)$, then 
    \[\omega h_\mu(T|R)+(1-\omega)h_{\pi\mu}(S|R)+\omega\int_Xfd\mu\leq\int_ZP_Z^\omega(\pi,T,z,f)d\kappa(z).\] 
    \item [\rm{(3)}] Therefore,  
    \[\omega h_\mu(T|R)+(1-\omega)h_{\pi\mu}(S|R)+\omega\int_Xfd\mu\leq\sup_{z\in Z}P_Z^\omega(\pi,T,z,f).\]
    \end{enumerate}
\end{proposition}
\begin{proof}
      We use a similar approach as in \cite{Tsu} by applying \textit{amplification trick}, that is, we shall prove that there are constants $C_0, C>0$ such that for any $k\in \mathbb{N}$,
    \[\omega h_\mu(T^k|R^k)+(1-\omega)h_{\pi\mu}(S^k|R^k)+\omega\int_X\mathbb{S}^T_kfd\mu\leq P^\omega_{Z}(\pi,T^k,\mathbb{S}_k^Tf)+C\] and for $\kappa=\psi\mu,$
   \[\omega h_\mu(T^k|R^k)+(1-\omega)h_{\pi\mu}(S^k|R^k)+\omega\int_X\mathbb{S}_k^Tfd\mu\leq\int_ZP_Z^\omega(\pi,T^k,z,\mathbb{S}_k^Tf)d\kappa(z)+C_0.\]
Since $h_{\{\cdot\}}(\cdot^k|\cdot^k)=kh_{\{\cdot\}}(\cdot|\cdot)$ and $\int_X\mathbb{S}_k^Tfd\mu=k\int_Xfd\mu$,
    then by Proposition \ref{prop2.2}, we have
    \[\omega h_\mu(T|R)+(1-\omega)h_{\pi\mu}(S|R)+\omega\int_Xfd\mu\leq P^\omega_{Z}(\pi,T,f)+\frac{C}{k}\]and
    \[\omega h_\mu(T|R)+(1-\omega)h_{\pi\mu}(S|R)+\omega\int_Xfd\mu\leq \int_ZP_Z^\omega(\pi,T,z,f)d\kappa(z)+\frac{C_0}{k}\]
    for all $k\in \mathbb{N}$.
    
(1) and (2) For each $\mathcal{A}\in \mathcal{P}_Y$ and $\mu\in \mathcal{M}(X,T)$, we write $\mathcal{A}^n=\bigvee_{i=0}^{n-1}S^{-i}\mathcal{A}$ and $\nu=\pi\mu$, $\kappa=\psi\mu.$

   At first, for any $\mathcal{A}=\{A_1,\cdots,A_\alpha\}\in \mathcal{P}_Y$ and $\mathcal{B}\in \mathcal{P}_X$  we will prove that 
     \[\omega h_\mu(T,\mathcal{B}|Z)+(1-\omega)h_{\pi\mu}(S,\mathcal{A}|Z)+\omega\int_Xfd\mu\leq P^\omega_{Z}(\pi,T,f)+C.\]
     
For each $1\leq i\leq\alpha$ we choose a compact subset $C_i\subset A_i$ such that 
 \begin{align}\label{eq3.1}
     \sum_{a=1}^\alpha\nu(A_i\backslash C_i)<\frac{1}{\log\alpha},
 \end{align}
 and set $C_0=Y\backslash (C_1\cup\cdots\cup C_\alpha)$ and $\mathcal{C}=\{C_0,C_1,\cdots,C_\alpha\}.$

 Let $\mathcal{B}\vee\pi^{-1}(\mathcal{C}).$ Suppose that it has the following form
 \[\mathcal{B}\vee\pi^{-1}(\mathcal{C})=\{B_{ij}|0\leq i\leq \alpha, 1\leq j\leq \beta_i\},~\pi^{-1}(C_i)=\bigcup_{j=1}^{\beta_i}B_{ij}~(0\leq i\leq \alpha).\]
For each $B_{ij}$ $(0\leq i\leq \alpha, 1\leq j\leq \beta_i)$, we take a compact subset $D_{ij}\subset B_{ij}$ such that 
\begin{align}\label{eqq2.2}
    \sum_{i=0}^\alpha\log\beta_i\left(\sum_{j=1}^{\beta_i}\mu(B_{ij}\backslash D_{ij})\right)<1.
\end{align}
We set 
\[D_{i0}=\pi^{-1}(C_i)\backslash\bigcup_{j=1}^{\beta_i}D_{ij}~(0\leq i\leq \alpha)\]
and define
\[\mathcal{D}=\{D_{ij}|0\leq i\leq \alpha,0\leq j\leq \beta_i\}.\]

   \begin{claim}\label{cl4.2}
      For each $n\in \mathbb{N}$, we have 
      \[H_\nu(\mathcal{A}^n|Z)\leq H_\nu(\mathcal{C}^n|Z)+nH_\nu(\mathcal{A}|\mathcal{C}).\]
      Hence,
      \[h_\nu(S,\mathcal{A}|Z)\leq h_\nu(S,\mathcal{C}|Z)+1.\]
    \end{claim}
    \begin{proof}
    \begin{align*}
    H_\nu(\mathcal{A}^n|Z)&\leq H_\nu(\mathcal{A}^n\vee\mathcal{C}^n|Z)= H_\nu(\mathcal{C}^n|Z)+H_\nu(\mathcal{A}^n|\mathcal{C}^n\vee\varphi^{-1}(\mathcal{B}_Z))\\&\leq H_\nu(\mathcal{C}^n|Z)+H_\nu(\mathcal{A}^n|\mathcal{C}^n)\\&\leq H_\nu(\mathcal{C}^n|Z)+nH_\nu(\mathcal{A}|\mathcal{C}).
\end{align*}
   Since $C_i\subset A_i$ for $1\leq i\leq \alpha$,
\[H_\nu(\mathcal{A}|\mathcal{C})=\nu(C_0)\sum_{i=1}^\alpha\left(-\frac{\nu(A_i\cap C_0)}{\nu(C_0)}\log\frac{\nu(A_i\cap C_0)}{\nu(C_0)}\right)\leq \nu(C_0)\log\alpha.\]
Thus,
\[h_\nu(S,\mathcal{A}|Z)\leq h_\nu(S,\mathcal{C}|Z)+1.~(\text{ by }\eqref{eq3.1})\]
    \end{proof}
 \begin{claim}\label{cl4.3}
  For each $n\in \mathbb{N}$, we have
  \[H_\mu(\mathcal{B}^n|Z)\leq H_\mu(\mathcal{D}^n|Z)+nH_\mu(\mathcal{B}\vee\pi^{-1}(\mathcal{C})|\mathcal{D}).\]
  Hence,
  \[h_\mu(T,\mathcal{B}|Z)\leq h_\mu(T,\mathcal{D}|Z)+1.\]
       \end{claim}
       \begin{proof}
           It is obvious that $\pi^{-1}(\mathcal{C}^{n})\preceq\mathcal{D}^{n}$ and it holds that
           \begin{align*}
               H_\mu(\mathcal{B}^n|Z)&\leq H_{\mu}((\mathcal{B}\vee\pi^{-1}(\mathcal{C}))^n\vee\mathcal{D}^n|Z)
               \\&\leq H_{\mu}(\mathcal{D}^{n}|Z)+H_{\mu}((\mathcal{B}\vee\pi^{-1}(\mathcal{C}))^n|\mathcal{D}^{n})\\&\leq H_{\mu}(\mathcal{D}^{n}|Z)+nH_\mu(\mathcal{B}\vee\pi^{-1}(\mathcal{C})|\mathcal{D}).
           \end{align*}
           Since $D_{ij}\subset B_{ij}$ for $0\leq i\leq \alpha$ and $1\leq j\leq \beta_i$, it holds that
           \begin{align*}
                &H_{\mu}(\mathcal{B}\vee\pi^{-1}(\mathcal{C})|\mathcal{D})\\&=\sum_{i=0}^\alpha\mu(D_{i0})\sum_{j=1}^{\beta_i}\left(-\frac{\mu(D_{i0}\cap B_{ij})}{\mu(D_{i0})}\log\frac{\mu(D_{i0}\cap B_{ij})}{\mu(D_{i0})}\right)\\&\leq 
                 \sum_{i=0}^\alpha \mu(D_{i0})\log\beta_i\,\,\text{ by }\eqref{eqq2.2}\\&<1.
           \end{align*}
         Thus,
         \[h_\mu(T,\mathcal{B}|Z)\leq h_\mu(T,\mathcal{D}|Z)+1.\]
       \end{proof}
Therefore, we obtain
\begin{align*}
    \omega h_\mu(T,\mathcal{B}|Z)+(1-\omega)h_\nu(S,\mathcal{A}|Z)\leq  \omega h_\mu(T,\mathcal{D}|Z)+(1-\omega)h_\nu(S,\mathcal{C}|Z)+2.
\end{align*}

For each $z\in Z$, we define
\[\mathcal{A}^{n,z}=\mathcal{A}^n\cap \varphi^{-1}(z)\text{ and }\mathcal{B}^{n,z}=\mathcal{B}^n\cap \psi^{-1}(z)\]
and 
\[\mathcal{C}^{n,z}=\mathcal{C}^n\cap \varphi^{-1}(z)\text{ and }\mathcal{D}^{n,z}=\mathcal{D}^n\cap\psi^{-1}(z).\]
Clearly, $\pi^{-1}(\mathcal{C}^{n,z})\preceq\mathcal{D}^{n,z}$.

Recall that $\nu(\cdot)=\mu(\pi^{-1}(\cdot))$ and for each $\mathcal{A}\in \mathcal{P}_Y$
\[H_\nu(\mathcal{A}|Z)=H_\mu(\pi^{-1}(\mathcal{A})|Z)=\int_Z H_{\mu_z}(\pi^{-1}(\mathcal{A}))d\kappa(z).\]
So we can write $\nu_z=\pi \mu_z$ for $\kappa$-a.e. $z\in Z$. Recall that $\mu_z$ and $\nu_z$ has full support on $\psi^{-1}(z)$ and $\varphi^{-1}(z)$, respectively. Therefore,
\begin{align*}
    \omega h_\mu(T,\mathcal{D}|Z)+(1-\omega)h_\nu(S,\mathcal{C}|Z)&=\lim_{n\to \infty}\left(\int_Z\omega\frac{H_{\mu_z}(\mathcal{D}^n)}{n}+(1-\omega)\frac{H_{\nu_z}(\mathcal{C}^n)}{n}d\kappa(z)\right)\\&
    =\lim_{n\to \infty}\left(\int_Z\omega\frac{H_{\mu_z}(\mathcal{D}^{n,z})}{n}+(1-\omega)\frac{H_{\nu_z}(\mathcal{C}^{n,z})}{n}d\kappa(z)\right),
\end{align*}
where $\mu=\int_Z\mu_zd\kappa(z)$ and $\nu_z=\pi\mu_z$.
Since $\pi^{-1}(\mathcal{C}^{n,z})\preceq\mathcal{D}^{n,z}$, we obtain
\begin{align*}
    \omega h_\mu(T,\mathcal{D}|Z)+(1-\omega)h_\nu(S,\mathcal{C}|Z)=\lim_{n\to \infty}\frac{1}{n}\left(\int_ZH_{\nu_z}(\mathcal{C}^{n,z})+\omega H_{\mu_z}(\mathcal{D}^{n,z}|\pi^{-1}(\mathcal{C}^{n,z}))d\kappa(z)\right).
\end{align*}

  For each $C\in \mathcal{C}^{n,z}$, we define
    \[\mathcal{D}^{n,z}_C=\{D\in \mathcal{D}^{n,z}|D\cap \pi^{-1}(C)\ne\varnothing\}=\{D\in \mathcal{D}^{n,z}|D\subset \pi^{-1}(C)\}.\]
   Then
    \[\pi^{-1}(C)=\bigsqcup_{D\in \mathcal{D}^{n,z}_C}D.\]
    For each $C\in \mathcal{C}^{n,z}$ with $\nu_z(C)>0$, and $D\in \mathcal{D}_C^{n,z}$, we write
    \[\mu_z(D|C)=\frac{\mu_z(D)}{\mu_z(\pi^{-1}(C))}=\frac{\mu_z(D)}{\nu_z(C)}.\]
    It is clear that 
    \[\sum_{D\in \mathcal{D}_C^{n,z}}\mu_z(D|C)=1.\]
    \begin{claim}\label{cl4.4}
       For $\kappa$-a.e. $z\in Z$ and $n\in \mathbb{N}$, we have 
        \[H_{\nu_z}(\mathcal{C}^{n,z})+\omega H_{\mu_z}(\mathcal{D}^{n,z}|\pi^{-1}(\mathcal{C}^{n,z}))+\omega\int_X\mathbb{S}_nfd\mu_z\leq\log\sum_{C\in \mathcal{C}^{n,z}}\left(\sum_{D\in \mathcal{D}_C^{n,z}}e^{\sup_{x\in D}\mathbb{S}_nf(x)}\right)^\omega.\]

    \end{claim}
    \begin{proof}
        We have
     \begin{align*}
    \int_X\mathbb{S}_nfd\mu_z&=\sum_{D\in \mathcal{D}^{n,z}}\int_{D}\mathbb{S}_nfd\mu_z\leq\sum_{D\in\mathcal{D}^{n,z}}\mu_z(D)\sup_{x\in D}\mathbb{S}_nf(x)\\&=\sum_{C\in\mathcal{C}^{n,z}}\nu_z(C)\left(\sum_{D\in \mathcal{D}_C^{n,z}}\mu_z(D|C)\sup_{x\in D}\mathbb{S}_nf(x)\right).
    \end{align*}
    By Lemma \ref{lem3.3}, we obtain
     \[\sum_{D\in \mathcal{D}_C^{n,z}}\left(-\mu_z(D|C)\log\mu_z(D|C)+\mu_z(D|C)\sup_{x\in D}\mathbb{S}_nf(x)\right)\leq\log\sum_{D\in\mathcal{D}_C^{n,z}}e^{\sup_{x\in D}}\mathbb{S}_nf(x).\]
     Hence, 
     \[H_{\mu_z}(\mathcal{D}^{n,z}|\pi^{-1}(\mathcal{C}^{n,z}))+\int_X\mathbb{S}_nfd\mu_z\leq \log\sum_{D\in\mathcal{D}_C^{n,z}}e^{\sup_{x\in D}}\mathbb{S}_nf(x).\]
     using Lemma \ref{lem3.3} again, it holds that
        \begin{align*}
            &H_{\nu_z}(\mathcal{C}^{n,z})+\omega H_{\mu_z}(\mathcal{D}^{n,z}|\pi^{-1}(\mathcal{C}^{n,z}))+\omega\int_X\mathbb{S}_nfd\mu_z\\&
            \leq \sum_{C\in\mathcal{C}^{n,z}}\left(-\nu_z(C)\log\nu_z(C)+\nu_z(C)\log\left(\sum_{D\in\mathcal{D}_C^{n,z}}e^{\sup_{x\in D}}\mathbb{S}_nf(x)\right)^\omega\right)\\&
            \leq \log\sum_{C\in\mathcal{C}^{n,z}}\left(\sum_{D\in\mathcal{D}_C^{n,z}}e^{\sup_{x\in D}}\mathbb{S}_nf(x)\right)^\omega.
        \end{align*}
    \end{proof}
     Let $d$ and $d'$ be metrics on $X$ and $Y$, respectively. Since $C_i\in \mathcal{C}$, $1\leq i\leq \alpha$ are mutually disjoint compact subset of $Y$ and $D_{ij}$, $0\leq i\leq \alpha$, $1\leq j\leq \beta_i$ are mutually disjoint compact subsets of $X$. Hence, we can find $\varepsilon>0$ (independent of the choice of $z\in Z$) such that for any $z\in Z$
    \begin{enumerate}
        \item  for any $y\in C_i^z(\subset C_i)\in \mathcal{C}^z$ and $y'\in C_{i'}^{z}(\subset C_{i'})\in \mathcal{C}^z$ with $i\ne i'\ne 0$,
        \[\varepsilon<d'(y,y');\]
        \item  for any $x\in D_{ij}^{z}(\subset D_{ij})\in \mathcal{D}^z$ and $x'\in D_{ij'}^{z}(\subset D_{ij'})\in \mathcal{D}^z$ with $j\ne j'\ne 0$,
        \[\varepsilon<d(x,x').\]
    \end{enumerate}
        \begin{claim}\label{cl4.5}
        For any $z\in Z$ and $n\in \mathbb{N}$:
        \begin{enumerate}
            \item [\rm{(1)}] If a subset $V\subset Y$ with $\diam(V,d'_n)<\varepsilon$, then the member of $C\in \mathcal{C}^{n,z}$ having non-empty intersection with $V$ at most $2^n$, namely,
            \[\left|\{C\in \mathcal{C}^{n,z}|C\cap V\ne\varnothing\}\right|\leq 2^n.\]
            \item [\rm{(2)}] If a subset $U\subset X$ with $\diam(U,d_n)<\varepsilon$, then for each $C\in \mathcal{C}^{n,z}$, the number of $D^{n,z}\in \mathcal{D}_C^{n,z}$ having non-empty intersection with $U$ is at most $2^n$:
            \[\left|\{D\in \mathcal{D}_{C}^{n,z}|D\cap U\ne\varnothing\}\right|\leq 2^n.\]
        \end{enumerate}
    \end{claim}
        \begin{proof}
       $(1)$ For each $0\leq k<n$, the set $S^kV$ may have non-empty intersection with $C_0^{z}$ and at most one set of $\{C_1^{z},\cdots,C_\alpha^{z}:C_i^z=C_i\cap \varphi^{-1}(z)\}$. Hence, the statement holds.

       $(2)$ Each $C\in \mathcal{C}^{n,z}$ has the form 
       \[C=C_{i_0}\cap S^{-1}C_{i_1}\cap S^{-2}C_{i_2}\cap \cdots\cap S^{-{n-1}}C_{i_{n-1}}\cap \varphi^{-1}(z),\]
       with $0\leq i_0,\cdots,i_{n-1}\leq\alpha$. Recall that $\{D_{i_k0},D_{i_k1},\cdots,D_{i_k\beta_{i_k}}\}$ is a partition of $\pi^{-1}(C_{i_k}).$ Then any set $D\in \mathcal{D}^{n,z}_C$ has the form 
       \[D=D_{i_0j_0}\cap T^{-1}D_{i_1j_1}\cap T^{-2}D_{i_2j_2}\cap \cdots\cap T^{-{n-1}}D_{i_{n-1}j_{n-1}}\cap \psi^{-1}(z),\]
       with $0\leq j_k\leq\beta_{i_k}$ for $0\leq k\leq n-1$.

       For each $0\leq k\leq n-1,$ the set $T^k U$ may have non-empty intersection with $D_{i_k0}$ and at most one set in $\{D^z_{i_k1},D^z_{i_k2},\cdots,D^z_{i_k\beta_{i_k}}:D^z_{i_k1}=D_{i_k1}\cap \psi^{-1}(z)\}$. The statement follows from this.
    \end{proof}
    
      Let $n\in \mathbb{N}$. Suppose there is an open cover $\{V_i^{n,z}\}_{i=1}^k\in \mathcal{C}^o_Y(\varphi^{-1}(z))$ with $\diam(V^{n,z}_i,d'_n)<\varepsilon$ for all $1\leq i\leq k$. Moreover, suppose that for each $1\leq i\leq k,$ there is an open cover $\{U_{ij}^{n,z}\}_{j=1}^{m_i}\in \mathcal{C}_X^o(\pi^{-1}(V_i^{n,z}))$ with $\diam(U^{n,z}_{ij},d_n)<\varepsilon$ for all $1\leq j\leq m_i$. For each $z\in Z$, we are going to prove that
    \begin{align}\label{eq4.2}
        \log\sum_{C\in\mathcal{C}^{n,z}}\left(\sum_{D\in \mathcal{D}_C^{n,z}}e^{\sup_D\mathbb{S}_nf}\right)^\omega\leq 2n\log2+\log\sum_{i=1}^k\left(\sum_{j=1}^{m_i}e^{\sup_{U^{n,z}_{ij}}\mathbb{S}_nf}\right)^\omega.
    \end{align}
    
  Indeed, suppose \eqref{eq4.2} is already proved. Then by Claim \ref{cl4.4}, for $\kappa$-a.e. $z\in Z$, we obtain
      \begin{align}\label{eq4}
        H_{\nu_z}(\mathcal{C}^{n,z})+\omega H_{\mu_z}(\mathcal{D}^{n,z}|\pi^{-1}(\mathcal{C}^{n,z}))+\omega\int_X\mathbb{S}_nfd\mu_z\leq 2n\log2+\log\sum_{i=1}^k\left(\sum_{j=1}^{m_i}e^{\sup_{x\in U^{n,z}_{ij}}\mathbb{S}_nf(x)}\right)^\omega.
    \end{align}
       Thus, by Claim \ref{cl4.2} and \ref{cl4.3}, we have
    \begin{align*}
        &\omega H_\mu(\mathcal{B}^n|Z)+(1-\omega)H_\nu(\mathcal{A}^n|Z)+\omega\int_X\mathbb{S}_nfd\mu\\&\leq \int_Z H_{\nu_z}(\mathcal{C}^{n,z})+\omega H_{\mu_z}(\mathcal{D}^{n,z}|\pi^{-1}(\mathcal{C}^{n,z}))d\kappa(z)+{\omega}\int_X\mathbb{S}_nfd\mu+2n\\&
        \leq 2n+2n\log 2+\int_Z\log P_Z^\omega(\pi,T,z,f,n,\varepsilon)d\kappa(z)\\&
        \leq 2n+2n\log2+\sup_{z\in Z}\log P_Z^\omega(\pi,T,z,f,n,\varepsilon),
    \end{align*}
    the second-to-last inequality is given by taking infimum over all $\{V^{n,z}_i\}$ and $\{U^{n,z}_{ij}\}$ satisfying \eqref{eq4}.
Then, divide the above inequality by $n$  and let $n$ to $\infty$, we have
    \begin{align*}
        &\omega h_\mu(T,\mathcal{B}|Z)+(1-\omega)h_\nu(S,\mathcal{A}|Z)+\omega\int_Xfd\mu\\&
        =\lim_{n\to \infty}\frac{1}{n}\left(\omega H_\mu(T,\mathcal{B}^n|Z)+(1-\omega)H_\nu(S,\mathcal{A}^n|Z)+\omega\int_X\mathbb{S}_nfd\mu\right)\\&
        \leq2+2\log 2+\limsup\frac{1}{n}\int_Z\log P_Z^\omega(\pi,T,z,f,n,\varepsilon)d\kappa(z)\\&
        \leq2+2\log 2+\int_Z\limsup\frac{1}{n}\log P_Z^\omega(\pi,T,z,f,n,\varepsilon)d\kappa(z)\\&
        \leq 2\log2+2+\lim_{n\to \infty}\left(\frac{\sup_{z\in Z}\log P^\omega_Z(\pi,T,z,f,n,\varepsilon)}{n}\right).
    \end{align*}
    Finally, taking $\varepsilon$ to $0$ and by Fatou's Lemma, we obtain
    \begin{align*}
        &\omega h_\mu(T,\mathcal{B}|Z)+(1-\omega)h_\nu(\mathcal{A}|Z)+\omega\int_Xfd\mu\\&
        \leq 2\log2+2+\int_Z P^\omega_Z(\pi,T,z,f)d\kappa(z)
        \\&\leq 2\log2+2+P^\omega_Z(\pi,T,f).
    \end{align*}
   Therefore, the rest is to prove \eqref{eq4.2}.

     Given $n\in \mathbb{N}$ and $z\in Z$. For each $D^{n,z}\in \mathcal{D}^{n,z}$, we have
   \[e^{\sup_{x\in D^{n,z}}\mathbb{S}_nf(x)}\leq \sum_{U_{ij}^{n,z}\cap D^{n,z}\ne\varnothing}e^{\sup_{x\in U_{ij}^{n,z}}\mathbb{S}_nf(x)}.\]
Here the sum is taken over all index $(i,j)$ such that $U_{ij}^{n,z}$ has non-empty intersection with $D^{n,z}.$

Let $C\in \mathcal{C}^{n,z}.$ We define $\mathcal{V}_C$ as the set of $1\leq i\leq k$ such that $V^{n,z}_i\cap C\ne\varnothing$. By Claim \ref{cl4.5}, we get 
    \[\sum_{D^{n,z}\in \mathcal{D}_C^{n,z}}e^{\sup_{D^{n,z}}\mathbb{S}_nf}\leq 2^n\sum_{i\in \mathcal{V}_C}\sum_{j=1}^{m_j}e^{\sup_{U_{ij}^{n,z}}\mathbb{S}_nf}.\]
    Therefore, 
    \begin{align*}
        \left(\sum_{D^{n,z}\in \mathcal{D}_C^{n,z}}e^{\sup_{x\in D^{n,z}}\mathbb{S}_nf(x)}\right)^\omega&\leq 2^{n\omega}\left(\sum_{i\in \mathcal{V}_C}\sum_{j=1}^{m_j}e^{\sup_{x\in U_{ij}^{n,z}}\mathbb{S}_nf(x)}\right)^\omega\\&
        \leq 2^{n\omega}\sum_{i\in \mathcal{V}_C}\left(\sum_{j=1}^{m_j}e^{\sup_{x\in U_{ij}^{n,z}}\mathbb{S}_nf(x)}\right)^\omega.
    \end{align*}
       \begin{remark}
        The last inequality holds since for $0\leq \omega\leq 1$ and non-negative numbers $x,y,$
\[(x+y)^\omega\leq x^\omega+y^\omega.\]
    \end{remark}
     Thus, 
    \[\sum_{C\in\mathcal{C}^{n,z}}\left(\sum_{D^{n,z}\in \mathcal{D}_C^{n,z}}e^{\sup_{x\in D^{n,z}}\mathbb{S}_nf(x)}\right)^\omega\leq 2^{n\omega}\sum_{C\in \mathcal{C}^{n,z}}\left(\sum_{i\in \mathcal{V}_C}\left(\sum_{j=1}^{m_i}e^{\sup_{x\in U^{n,z}_{ij}}\mathbb{S}_nf(x)}\right)^\omega\right).\]
    By Claim \ref{cl4.5}, for each $1\leq i\leq k$, the number $C\in \mathcal{C}^{n,z}$ satisfying $i\in \mathcal{V}_C$ is at most $2^n.$ So
\[2^{n\omega}\sum_{C\in \mathcal{C}^{n,z}}\left(\sum_{i\in \mathcal{V}_C}\left(\sum_{j=1}^{m_j}e^{\sup_{x\in U^{n,z}_{ij}}\mathbb{S}_nf(x)}\right)^\omega\right)\leq 2^{n\omega}\cdot 2^n\sum_{i=1}^k\left(\sum_{j=1}^{m_i}e^{\sup_{x\in U^{n,z}_{ij}}\mathbb{S}_nf(x)}\right)^\omega.\]
Therefore, 
\[\sum_{C\in\mathcal{C}^{n,z}}\left(\sum_{D^{n,z}\in \mathcal{D}_C^{n,z}}e^{\sup_{x\in D^{n,z}}\mathbb{S}_nf(x)}\right)^\omega\leq 
2^{n\omega}\cdot 2^n\sum_{i=1}^k\left(\sum_{j=1}^{m_i}e^{\sup_{x\in U^{n,z}_{ij}}\mathbb{S}_nf(x)}\right)^\omega.\]
Taking the logarithm,
\begin{align*}
    \log\sum_{C\in\mathcal{C}^{n,z}}\left(\sum_{D^{n,z}\in \mathcal{D}_C^{n,z}}e^{\sup_{x\in D^{n,z}}\mathbb{S}_nf(x)}\right)^\omega&\leq (n+n\omega)\log2+\log\sum_{i=1}^k\left(\sum_{j=1}^{m_i}e^{\sup_{x\in U^{n,z}_{ij}}\mathbb{S}_nf(x)}\right)^\omega\\&
    \leq 2n\log2+\sum_{i=1}^k\left(\sum_{j=1}^{m_i}e^{\sup_{x\in U^{n,z}_{ij}}\mathbb{S}_nf(x)}\right)^\omega,
\end{align*}
which proves \eqref{eq4.2}. Therefore, we finish the proof.

(3) As $\psi$ maps $\mathcal{M}(X,T)$ onto $\mathcal{M}(Z,R)$, the result follows directly from (2).
\end{proof}

\subsection{Proof of the other side of variational principles}
 \,
 
Recall that we write
\[P^\omega_{Z,var}(\pi,T,f)=\sup_{\mu\in \mathcal{M}(X,T)}\left(\omega h_\mu(T|R)+(1-\omega)h_{\pi\mu}(S|R)+\omega\int_Xfd\mu\right),\] then we state the following result.

\begin{proposition}\label{prop4.11}
 Let $\pi:(X,T)\to (Y,S)$ be a factor map between zero-dimensional TDSs and suppose $\varphi$ is a factor map from $(Y, S)$ to $(Z, R)$. Then for any $0\leq \omega\leq 1$ and $f\in C(X)$, we have
    \begin{enumerate}
        \item [\rm{(1)}] $P^\omega_{Z}(\pi,T,f)\leq P^\omega_{Z,var}(\pi,T,f);$
        \item [\rm{(2)}] $\sup_{z\in Z}P_Z^\omega(\pi,T,z,f)\leq P^\omega_{Z,var}(\pi,T,f).$
    \end{enumerate}
\end{proposition}
\begin{proof}
    Let $\varepsilon>0$ and $\mathcal{A}$ be a clopen partition of $Y$ with $\diam(\mathcal{A},d')<\varepsilon$ and denote $\mathcal{A}^n=\bigvee_{i=0}^{n-1}S^{-i}\mathcal{A}$. For each $z\in Z$ and $n\in \mathbb{N}$, we set 
\[\mathcal{A}^{n,z}=\{A\cap \varphi^{-1}(z)|A\in \mathcal{A}^n\}.\]
Since $X$ is zero-dimensional, for each $1\leq i\leq\alpha$, we can take a clopen partition $\mathcal{B}=\{B_{ij}\}\in\mathcal{P}_X$ such that  $\diam(\mathcal{B},d)<\varepsilon$ and $\pi^{-1}(\mathcal{A})\preceq \mathcal{B}$, also, we write $\mathcal{B}^n=\bigvee_{i=0}^{n-1}T^{-i}\mathcal{B}.$. For each $A_i\in \mathcal{A}$ we have
\[\pi^{-1}(A_i)=\bigsqcup_{j=1}^{\beta_{ij}}B_{ij}.\]

We put
\[\mathcal{B}^{n,z}=\{B\cap \psi^{-1}(z):B\in \mathcal{B}^n\}=\{B\cap \pi^{-1}(A^{n,z})|A^{n,z}\in \mathcal{A}^{n,z}, B\in \mathcal{B}^{n}\},\]
then $\mathcal{B}^{n,z}$ is a partition of $\psi^{-1}z$ and each $A^{n,z}$ is a disjoint union of some $B^{n,z}\in \mathcal{B}^{n,z}$. For each $A^{n,z}\in \mathcal{A}^{n,z}$, we define
\begin{align}\label{eq4.5}
    \mathcal{B}_{A^{n,z}}^{n,z}=\{B^{n,z}\in\mathcal{B}^{n,z}|B^{n,z}\cap\pi^{-1}(A^{n,z})\ne\varnothing\}=\{B^{n,z}\in \mathcal{B}^{n,z}|B^{n,z}\subset \pi^{-1}(A^{n,z})\}.
\end{align}
So we have
\[\pi^{-1}(A^{n,z})=\bigsqcup_{B^{n,z}\in \mathcal{B}_{A^{n,z}}^{n,z}}B^{n,z}.\]
For each $n\in \mathbb{N}$ and $z\in Z$, we set
\[W_{A^{n,z}}=\sum_{B^{n,z}\in\mathcal{B}_{A^{n,z}}^{n,z}}e^{\sup_{x\in B^{n,z}}\mathbb{S}_nf(x)}\]
and define
\[W_{n,z}=\sum_{A^{n,z}\in \mathcal{A}^{n,z}}\left(W_{A^{n,z}}\right)^\omega.\]
Then, from the definition, we have the following property
    \[P^\omega_Z(\pi,T,z,f,n,\varepsilon)\leq W_{n,z}.\]

Let $\varepsilon_0>0$ small enough. For each $n\in\mathbb{N}$ we choose a point $z'\in Z$ such that 
\begin{align}\label{eq4.6}
    \sup_{z\in Z}\log P_Z^\omega(\pi,T,z,f,n,\varepsilon)\leq \log P_Z^\omega(\pi,T,z',f,n,\varepsilon)+\varepsilon_0.
\end{align}

Now, fix $n\in \mathbb{N}$ and let $z\in Z$ be a point satisfying condition \eqref{eq4.6}. We assume that the elements of $\mathcal{B}^{n,z}$ and $\mathcal{A}^{n,z}$ are all non-empty. For each $B^{n,z}\in \mathcal{B}^{n,z}$, we denote by $\mathcal{A}^{n,z}(B^{n,z})$ the unique element in $\mathcal{A}^{n,z}$ containing $\pi(B^{n,z}).$
Since each $B^{n,z}\in \mathcal{B}^{n,z}$ is compact, we can take a point $x_{B^{n,z}}\in B^{n,z}$ satisfying $\mathbb{S}_nf(x_{B^{n,z}})=\sup_{x\in B^{n,z}}\mathbb{S}_nf(x).$ and we define a probability measure on $X$ by 
\begin{align*}
    \sigma_{n}&=\frac{1}{W_{n,z}}\sum_{B^{n,z}\in \mathcal{B}^{n,z}}(W_{\mathcal{A}^{n,z}(B^{n,z})})^{\omega-1}e^{\mathbb{S}_nf(x_{B^{n,z}})}\cdot \delta_{x_{B^{n,z}}}\\& =\frac{1}{W_{n,z}}\sum_{A^{n,z}\in \mathcal{A}^{n,z}}\sum_{B^{n,z}\in \mathcal{B}^{n,z}_{A^{n,z}}}(W_{A^{n,z}})^{\omega-1}e^{\mathbb{S}_nf(x_{B^{n,z}})}\cdot \delta_{x_{B^{n,z}}},
\end{align*}
where $\delta_{x_{B^{n,z}}}$ is the probability measure mass on the point $x_{B^{n,z}}.$ We set 
\begin{align*}
    \mu_{n}=\frac{1}{n}\sum_{s=0}^{n-1}T^s\sigma_{n}.
\end{align*}
 We can take a subsequence $\{\mu_{n_k}\}$ converging to an invariant measure $\mu\in \mathcal{M}(X,T),$ then we shall prove that 
\begin{align*}
    &\omega h_\mu(T,\mathcal{B}|Z)+(1-\omega)h_{\pi\mu}(S,\mathcal{A}|Z)+\omega\int_Xfd\mu\\&\geq \lim_{n\to \infty}\frac{\log W_{n,z}}{n}\\&\geq\lim_{n\to \infty}\frac{\sup_{z\in Z}\log \ P^\omega_Z(\pi,T,z,f,n,\varepsilon)-\varepsilon_0}{n}.  
\end{align*}
\begin{claim}\label{cl4.13}
 For any natural number $n\in \mathbb{N}$, let $z\in Z$ be a point satisfying condition \eqref{eq4.6} and $\sigma_n$ is the probability measure defined as above,  we have 
    \begin{align*}
        &\omega H_{\sigma_n}(\mathcal{B}^n|Z)+(1-\omega)H_{\pi\sigma_n}(\mathcal{A}^n|Z)+\omega\int_X \mathbb{S}_nfd\sigma_n\\&=\log W_{n,z}\\&\geq\sup_{z\in Z}\log P^\omega_Z(\pi,T,z,f,n,\varepsilon)-\varepsilon_0.
    \end{align*}
\end{claim}
\begin{proof}
     From the construction of the probability measure $\sigma_n,$ for each $B\in \mathcal{B}^n,$ if $B\cap \psi^{-2}(z)\ne\emptyset$, we have 
   \[ \sigma_n(B)=\sigma_n(B^{n,z})=\frac{(W_{\mathcal{A}^{n,z}(B^{n,z})})^{\omega-1}}{W_{n,z}}e^{\mathbb{S}_nf(x_{B^{n,z}})}.
    \]
    Otherwise, $\sigma_n(B)=0$ and we assume that $0\log 0=0$. Then 
         \begin{align*}
         H_{\sigma_n}(\mathcal{B}^n|Z)&=H_{\sigma_n}(\mathcal{B}^{n,z}|Z)\\&
         =-\sum_{B^{n,z}\in \mathcal{B}^{n,z}}\frac{(W_{\mathcal{A}^{n,z}(B^{n,z})})^{\omega-1}}{W_{n,z}}e^{\mathbb{S}_nf(x_{B^{n,z}})}\log\left(\frac{(W_{\mathcal{A}^{n,z}(B^{n,z})})^{\omega-1}}{W_{n,z}}e^{\mathbb{S}_nf(x_{B^{n,z}})}\right)\\&=\frac{\log W_{n,z}}{W_{n,z}}\underbrace{\sum_{B^{n,z}\in \mathcal{B}^{n,z}}(W_{\mathcal{A}^{n,z}(B^{n,z})})^{\omega-1}e^{\mathbb{S}_nf(x_{B^{n,z}})}}_{(I)}\\& -\frac{\omega-1}{W_{n,z}}\underbrace{\sum_{B^{n,z}\in \mathcal{B}^{n,z}}(W_{\mathcal{A}^{n,z}(B^{n,z})})^{\omega-1}e^{\mathbb{S}_nf(x_{B^{n,z}})}\log W_{\mathcal{A}^{n,z}(B^{n,z})}}_{(II)}\\&
         -\underbrace{\sum_{B^{n,z}\in\mathcal{B}^{n,z}}\frac{(W_{\mathcal{A}^{n,z}(B^{n,z})})^{\omega-1}}{W_{n,z}}e^{\mathbb{S}_nf(x_{B^{n,z}})}\mathbb{S}_nf(x_{B^{n,z}})}_{(III)}.
     \end{align*}
     Based on the definition, we obtain that 
     the term $(I)$ can be calculated by 
     \[(I)=\sum_{A^{n,z}\in \mathcal{A}^{n,z}}\sum_{B^{n,z}\in\mathcal{B}^{n,z}_{A^{n,z}}}(W_{A^{n,z}})^{\omega-1}e^{\mathbb{S}_nf(x_{B^{n,z}})}=W_{n,z}.\]
     Term $(II)$ is obtained by 
     \[(II)=\sum_{A^{n,z}\in\mathcal{A}^{n,z}}\sum_{B^{n,z}\in\mathcal{B}^{n,z}_{A^{n,z}}}(W_{A^{n,z}})^{\omega-1}e^{\mathbb{S}_nf(x_{B^{n,z}})}\log W_{A^{n,z}}=\sum_{A^{n,z}\in\mathcal{A}^{n,z}}(W_{A^{n,z}})^\omega\log W_{A^{n,z}}.\]
     For term $(III)$, we have
     \[\int_X\mathbb{S}_nfd\sigma_n=\frac{1}{W_{n,z}}\sum_{B^{n,z}\in\mathcal{B}^{n,z}}(W_{\mathcal{A}^{n,z}(B^{n,z})})^{\omega-1}e^{\mathbb{S}_nf(x_{B^{n,z}})}\mathbb{S}_nf(x_{B^{n,z}})=(III).\]
     Therefore, 
     \begin{align}\label{eq4.7}
         H_{\sigma_n}(\mathcal{B}^n|Z)+\int_X\mathbb{S}_nfd\sigma_n=\log W_{n,z}-\frac{\omega-1}{W_{n,z}}\sum_{A^{n,z}\in\mathcal{A}^{n,z}}(W_{A^{n,z}})^\omega\log W_{A^{n,z}}.
     \end{align}
     Moreover, we have
     \[\pi\sigma_n=\frac{1}{W_{n,z}}\sum_{B^{n,z}\in \mathcal{B}^{n,z}}(W_{\mathcal{A}^{n,z}(B^{n,z})})^{\omega-1}e^{\mathbb{S}_nf(x_{B^{n,z}})}\cdot\delta_{\pi(x_{B^{n,z}})}.\]
     From the construction of $\sigma_n$, for each non-empty $A\in \mathcal{A}^{n},$ $A^{n,z}\subset A\cap \varphi^{-1}(z)$ we have
     \begin{align*}
         \pi\sigma_n(A)=\pi\sigma_n(A^{n,z})&=\frac{1}{W_{n,z}}\sum_{B^{n,z}\in\mathcal{B}^{n,z}_{A^{n,z}}}(W_{\mathcal{A}^{n,z}(B^{n,z})})^{\omega-1}e^{\mathbb{S}_nf(x_{B^{n,z}})}\\&
         =\frac{1}{W_{n,z}}(W_{A^{n,z}})^\omega,
     \end{align*}
     where $\mathcal{A}^{n,z}(B^{n,z})=A^{n,z}$ for $B^{n,z}\in \mathcal{B}^{n,z}_{A^{n,z}}$. Then
     \[H_{\pi\sigma_n}(\mathcal{A}^{n}|Z)=\log W_{n,z}-\omega\sum_{A^{n,z}\in \mathcal{A}^{n,z}}\frac{(W_{A^{n,z}})^\omega}{W_{n,z}}\log W_{A^{n,z}}.\]
      Combining this with \eqref{eq4.7}, we obtain
\[\omega H_{\sigma_n}(\mathcal{B}^n|Z)+(1-\omega)H_{\pi\sigma_n}(\mathcal{A}^n|Z)+\omega\int_X\mathbb{S}_nfd\sigma_n=\log W_{n,z}.\]
\end{proof}
The proof of the following claim is standard (See the proof of the variational principle in \cite{Pw}), but for the sake of completeness, we will write it out.
\begin{claim}\label{cl4.14}
    Let $m<n$ be positive integers. We have
    \[\frac{1}{m}H_{\mu_n}(\mathcal{B}^m|Z)\geq\frac{1}{n}H_{\sigma_n}(\mathcal{B}^n|Z)-\frac{2m\log|\mathcal{B}|}{n},\]
    \[\frac{1}{m}H_{\pi\mu_n}(\mathcal{A}^m|Z)\geq\frac{1}{n}H_{\pi\sigma_n}(\mathcal{A}^n|Z)-\frac{2m\log|\mathcal{A}|}{n},\]
    where $|\cdot|$ is the cardinal operator.
\end{claim}
\begin{proof}
 Here, we provide the proof for $\mathcal{B}^m$, the case of $\mathcal{A}^m$ is similar.
 We assume that $1<m<n$, and for $0\leq l<m,$ let $a(l)$ denote the integer part of $(n-l)m^{-1}$, so that $n=l+a(l)m+r$ with $0\leq r<q$. Then
 \[\mathcal{B}^n=\bigvee_{i=0}^{n-1}T^{-i}\mathcal{B}=\left(\bigvee_{j=0}^{a(l)-1}T^{-(l+jm)}\bigvee_{i=0}^{m-1}T^{-i}\mathcal{B}\right)\vee\bigvee_{t\in S_l}T^{-t}\mathcal{B},\]
 where $S_l$ is a subset of $\{0,1,\cdots,n-1\}$ with cardinality at most $2m$. Then we obtain
 \begin{align*}
     H_{\sigma_n}\left(\bigvee_{i=0}^{n-1}T^{-i}\mathcal{B}|Z\right)&\leq\sum_{j=0}^{a(l)-1}H_{\sigma_n}\left(T^{-(l+jm)}\bigvee_{i=0}^{m-1}T^{-i}\mathcal{B}|Z\right)+2m\log|\mathcal{B}|\\&\leq \sum_{j=0}^{a(l)-1}H_{T^{(l+jm)}\sigma_n}\left(\bigvee_{i=0}^{m-1}T^{-i}\mathcal{B}|Z\right)+2m\log|\mathcal{B}|.
 \end{align*}
 Sum this inequality over $l\in \{0,1,\cdots,m-1\},$ we have that 
 \begin{align*}
     mH_{\sigma_n}\left(\bigvee_{i=0}^{n-1}T^{-i}\mathcal{B}|Z\right)&\leq \sum_{t=0}^{n-1}H_{T^t\sigma_n}\left(\bigvee_{i=0}^{m-1}T^{-i}\mathcal{B}|Z\right)+2m^2\log|\mathcal{B}|\\&
     \leq nH_{\mu_n}\left(\bigvee_{i=0}^{m-1}T^{-i}\mathcal{B}|Z\right)+2m^2\log|\mathcal{B}|,
 \end{align*}
 where the second inequality depends on the general property of the conditional entropy of partitions $H_{\sum_ip_i\mu_i}(\mathcal{B}|\mathcal{R})\geq \sum_{i}p_iH_{\mu_i}(\mathcal{B}|\mathcal{R})$ which holds for any finite partition $\mathcal{B}$, $\sigma$-algebra $\mathcal{R}$, Borel probability measures $\mu_i$, and positive numbers $p_i$ with $p_1+\cdots+p_n=1$ (See Lemma \cite[Lemma 3.2]{Lp}). Dividing by $nm$ in the above inequality, we obtain
 \[\frac{1}{n}H_{\sigma_n}\left(\bigvee_{i=0}^{n-1}T^{-i}\mathcal{B}|Z\right)\leq \frac{1}{m}H_{\mu_n}\left(\bigvee_{i=0}^{m-1}T^{-i}\mathcal{B}|Z\right)+\frac{2m}{n}\log|\mathcal{B}|,\]
 which completes the proof of the claim.
\end{proof}
From the construction of $\mu_n$, we have
\[\int_Xfd\mu_n=\frac{1}{n}\int_X\sum_{i=0}^{n-1}f\circ T^id\sigma_n=\frac{1}{n}\int_X\mathbb{S}_nfd\sigma_n.\]
Then Claim \ref{cl4.14} implies that
\begin{align*}
    \frac{\omega}{m}&H_{\mu_n}(\mathcal{B}^m|Z)+\frac{1-\omega}{m}H_{\pi\mu_n}(\mathcal{A}^m|Z)+\omega\int_Xfd\mu_n\\&
    \geq\frac{\omega}{n}H_{\sigma_n}(\mathcal{B}^n|Z)+\frac{1-\omega}{n}H_{\sigma_n}(\mathcal{A}^n|Z)+\frac{\omega}{n}\int_Xfd\sigma_n-\frac{2m(\log|\mathcal{A}|\cdot|\mathcal{B}|)}{n}\\&=\frac{\log W_{n,z}}{n}-\frac{2m(\log|\mathcal{A}|\cdot|\mathcal{B}|)}{n}\\&\geq \frac{\sup_{z\in Z}\log P^\omega_Z(\pi,T,z,f,n,\varepsilon)-\varepsilon_0}{n}-\frac{2m(\log|\mathcal{A}|\cdot|\mathcal{B}|)}{n},
\end{align*}
where the last inequality is obtained from Claim \ref{cl4.13}.

For each $n\in \mathbb{N}$ the boundary of $\mathcal{A}^n$ and $\mathcal{B}^n$ has measure zero, so by taking $\mu_{n_k}\to \mu$ as $k\to \infty$, we have  
\begin{align*}
    &\frac{\omega}{m}H_{\mu}(\mathcal{B}^m|Z)+\frac{1-\omega}{m}H_{\pi\mu}(\mathcal{A}^m|Z)+\omega\int_Xfd\mu\\&\geq\lim_{n\to \infty}\frac{\sup_{z\in Z}\log P^\omega_Z(\pi,T,z,f,n,\varepsilon)-\varepsilon_0}{n}.
\end{align*}
Finally, let $m\to \infty$ and $\varepsilon_0\to 0$. We get 
\[\omega h_\mu(T,\mathcal{B}|Z)+(1-\omega)h_{\pi\mu}(S,\mathcal{A}|Z)+\omega\int_Xfd\mu\geq\lim_{n\to \infty}\frac{\sup_{z\in Z}\log P^\omega_Z(\pi,T,z,f,n,\varepsilon)}{n}.\]
Since for each $n\in \mathbb{N}$ and $z_0\in Z$,
\[\frac{\sup_{z\in Z}\log P^\omega_Z(\pi,T,z,f,n,\varepsilon)}{n}\geq \frac{\log P^\omega_Z(\pi,T,z_0,f,n,\varepsilon)}{n},\]
we also obtain
\[\omega h_\mu(T,\mathcal{B}|Z)+(1-\omega)h_{\pi\mu}(S,\mathcal{A}|Z)+\omega\int_Xfd\mu\geq \sup_{z\in Z}P_Z^\omega(\pi,T,z,f,\varepsilon).\]

Therefore, (1) and (2) are obtained by taking $\mathcal{A}$ and $\mathcal{B}$ with $\diam(\mathcal{A},d')\to 0$ and $\diam(\mathcal{B},d)\to 0.$

\end{proof}

\begin{theorem}
    Let  $\pi:(X,T)\to (Y,S)$ be a factor map between TDSs, $(Z,R)$ is a factor of $(Y,S)$ via $\varphi$ and $f\in C(X)$. Then for any $0\leq \omega\leq 1$, we have 
    \begin{enumerate}
        \item [\rm{(1)}] \[P^\omega_Z(\pi,T,f)=\sup_{\mu\in \mathcal{M}(X,T)}\left(\omega h_\mu(T|R)+(1-\omega)h_{\pi\mu}(S|R)+\omega\int_Xfd\mu\right).\]
        \item [\rm{(2)}] \[\sup_{z\in Z}P_Z^\omega(\pi,z,f)= \sup_{\mu\in \mathcal{M}(X,T)}\left(\omega h_\mu(T|R)+(1-\omega)h_{\pi\mu}(S|R)+\omega\int_Xfd\mu\right).\]
    \end{enumerate}
    Therefore, 
    \[P^\omega_Z(\pi,T,f)=\sup_{z\in Z}P_Z^\omega(\pi,z,f).\]

\end{theorem}
\begin{proof}
  It directly follows from Proposition \ref{prop4.10} and Proposition \ref{prop4.11}.
\end{proof}

\begin{proposition}\label{prop4.16}
    Let $\kappa\in \mathcal{M}(Z,R)$ and let $z$ be a generic point of $\kappa$, namely,
    \begin{align*}
        \frac{1}{n}\sum_{i=0}^{n-1}\delta_{R^iz}\to  \kappa \text{ as }n\to \infty.
    \end{align*}
    Let $\varepsilon>0$. There exists $\mu\in \mathcal{M}(X,T)$ with $\nu=\pi\mu$ and $\kappa=\psi\mu$ such that
\begin{align*}
    P^\omega_Z(\pi,T,z,f,\varepsilon)\leq \omega h_\mu(T|R)+(1-\omega)h_\nu(S|R)+\omega\int_Xfd\mu.
\end{align*}
\end{proposition}
\begin{proof}
    Because of Proposition \ref{prop4.10}, we just need to prove the conclusion that both $X$ and $Y$ are zero-dimensional.
From the construction of $\mu_n$ in Proposition \ref{prop4.11}, it is clear that if $\mu$ is a limit point of $\mu_n$, then $\psi\mu=\kappa$ as $z$ is a generic point of $\kappa$. Therefore, in this proposition, it suffices to take the $\mu$ constructed in Proposition \ref{prop4.11}.
    
\end{proof}

For a TDS $(X,T)$, we denote $E(X,T)$ by the set of all $T$-invariant ergodic measures on $(X,T)$.

\begin{theorem}
   Let $\pi:(X,T)\to (Y,S)$ and $\varphi:(Y,S)\to (Z,R)$ with $\psi=\varphi\circ\pi$, $f\in C(X)$ and $\kappa\in \mathcal{M}(Z,R)$. Given $0\leq \omega\leq 1$, we have 
    \begin{align*}
        \int_ZP_Z^\omega(\pi,T,z,f)d\kappa(z)=\sup\left(\omega h_\mu(T|R)+(1-\omega)h_{\pi\mu}(S|R)+\omega\int_Xfd\mu\right)
    \end{align*}
    where the supremum is taken over all $\mu\in \mathcal{M}(X,T)$ with $\kappa=\psi\mu$.
\end{theorem}
\begin{proof}
    The proof follows a similar procedure as in \cite{Lp}. 
    On the one hand,
    1. suppose $\kappa$ is ergodic, that is, $\kappa\in E(Z,R)$. Let $z$ be a generic point of $\kappa$ and $\varepsilon>0$. By Proposition \ref{prop4.16},
    \[P^\omega_Z(\pi,T,z,f,\varepsilon)\leq \sup_{\psi\mu=\kappa}\left(\omega h_\mu(T|R)+(1-\omega)h_{\pi\mu}(S|R)+\omega\int_Xfd\mu\right)=a.\]
    Since $\kappa$-a.e. $z\in Z$ are generic, we have
    \[\int_ZP^\omega_Z(\pi,T,z,f,\varepsilon)d\kappa(z)\leq a.\]
    But $P^\omega_Z(\pi,T,z,f,\varepsilon)\nearrow P^\omega_Z(\pi,T,z,f)$ as $\varepsilon\to 0$ and the function $P^\omega_Z(\pi,T,z,f,\varepsilon)$ are clearly bounded from below by $-||f||$. So $\int_ZP^\omega_Z(\pi,T,z,f)d\kappa(z)\leq a$ if $\kappa$ is ergodic.

     2. If $\kappa$ is not ergodic, let $\kappa=\int_{E(Z,R)}\kappa_\alpha d\rho(\alpha)$ be its ergodic decomposition.
    Let $\delta>0$, define
    \begin{align*}
        K_\delta=\{(\tau,\mu)&\in E(Z,R)\times\mathcal{M}(X,T)|\psi\mu=\tau,\\&\omega h_\mu(T|R)+(1-\omega)h_{\pi\mu}(S|R)+\omega\int_Xfd\mu\geq \int_ZP^\omega_Z(\pi,T,z,f)d\kappa(z)-\delta\}.
    \end{align*}
   Then $K_\delta$ is a measurable subset of $E(Z, R)\times\mathcal{M}(X, T)$, and we have shown above that $K_\delta$ projects onto $E(Z, R)$. Hence there is a section $K_\delta$, that is, a measurable map $\phi_\delta:E(Z,R)\to \mathcal{M}(X,T)$ such that 
    \[\rho\left(\{\tau|(\tau,\phi_\delta(\tau))\in K_\delta\}\right)=1.\]
    Define $\mu_\delta$ by $\mu_\delta=\int \phi_\delta(\kappa_\alpha)d\rho(\alpha).$

    Then $\mu_\delta\in \mathcal{M}(X,T),$ $\nu_\delta=\pi\mu_\delta$ and $\kappa=\psi\mu_\delta$ satisfying 
    \begin{align*}
        &\omega h_{\mu_\delta}(T|R)+(1-\omega)h_{\nu_\delta}(S|R)+\omega\int_Xfd\mu
        \\&=\int \omega h_{\phi_\delta(\kappa_\alpha)}(T|R)+(1-\omega)h_{\pi\phi_\delta(\kappa_\alpha)}(S|R)d\rho(\alpha)+\omega\int\left(\int_X fd\phi_\delta(\kappa_\alpha)\right))d\rho(\alpha)
        \\&\geq \int \left(\int P^\omega_Z(\pi,T,z,f)d\kappa(z)-\delta \right)d\rho(\alpha)
        \\&=\int_Z P^\omega_Z(\pi,T,z,f)d\kappa(z)-\delta.
    \end{align*}
Therefore, 
\[\int_Z P^\omega_Z(\pi,T,z,f)d\kappa(z)\leq \sup_{\psi\mu=\kappa}\left(\omega h_\mu(T|R)+(1-\omega)h_{\pi\mu}(S|R)+\omega\int_Xfd\mu\right).\]

 On the other hand, from Proposition \ref{prop4.1} we have known that 
    \[\int_Z P^\omega_Z(\pi,T,z,f)d\kappa\geq \sup_{\psi\mu=\kappa}\left(\omega h_\mu(T|R)+(1-\omega)h_{\pi\mu}(S|R)+\omega\int_Xfd\mu\right).\]
    Then the conclusion follows.
\end{proof}

\section*{Acknowledgement}
I thank Professor Xiongping Dai and Professor Dou Dou for their helpful discussions. The author also would like to thank the anonymous reviewers for their insightful comments.

\raggedright
\raggedright
\scriptsize{\textbf{Conflicts Statement}
The author declares that there are no conflicts of interest.

\textbf{Data Availability} All data generated or analyzed during this study are included in this article (and its supplementary information files).}

\vspace{0.5cm}

\address{School of Mathematics, Nanjing University, Nanjing 210093, People's Republic of China}

\textit{E-mail}: \texttt{yzy\_nju\_20@163.com}


\begin{thebibliography}{99}
\bibitem{Akm}
\newblock R. L. Adler, A. G. Konheim and M. H. McAndrew,
\newblock \textit{Topological entropy},
\newblock Trans. Amer. Math. Soc., \textbf{114} (1965), 309--319.

\bibitem{Bed}
\newblock T. Bedford,
\newblock \textit{Crinkly curves. Markov partitions and box dimension in self-similar sets},
\newblock Ph.D. Thesis, Univ. Warwick, 1984.

\bibitem{DZ}
\newblock A. H. Dooley and G. H. Zhang,
\newblock \textit{Local entropy of a random dynamical system},
\newblock Mem. Amer. Math. Soc., \textbf{233} (2015), 1--120.

\bibitem{Dow}
\newblock T. Downarowicz,
\newblock \textit{Entropy in dynamical systems},
\newblock Cambridge Univ. Press, 2011.

\bibitem{DH}
\newblock T. Downarowicz and D. Huczek,
\newblock \textit{Zero-dimensional principal extensions},
\newblock Acta Appl. Math., \textbf{126} (2013), 117--129.

\bibitem{DS}
\newblock T. Downarowicz and J. Serafin,
\newblock \textit{Fiber entropy and conditional variational principles in compact non-metrizable space},
\newblock Fund. Math., \textbf{172} (2002), 217--247.

\bibitem{FH}
\newblock D. Feng and W. Huang,
\newblock \textit{Variational principle for weighted topological pressure},
\newblock J. Math. Pures Appl., \textbf{106} (2016), 411--452.

\bibitem{HYZ}
\newblock W. Huang, X. Ye and G. Zhang,
\newblock \textit{A local variational principle for conditional entropy},
\newblock Ergod. Theory Dyn. Syst., \textbf{26} (2006), 219--245.

\bibitem{HYZ1}
\newblock W. Huang, X. Ye and G. Zhang,
\newblock \textit{Relative entropy tuples, relative u.p.e. and c.p.e. extensions},
\newblock Israel J. Math., \textbf{158} (2007), 249--283.

\bibitem{KP}
\newblock R. Kenyon and Y. Peres,
\newblock \textit{Measures of full dimension on affine-invariant sets},
\newblock Ergod. Theory Dyn. Syst., \textbf{16} (1996), 307--323.

\bibitem{Ke}
\newblock J. L. Kelley,
\newblock \textit{General topology},
\newblock Springer, 1955.

\bibitem{Lp}
\newblock F. Ledrappier and P. Walters,
\newblock \textit{A relativized variational principle for continuous transformations},
\newblock J. London Math. Soc., \textbf{16} (1977), 568--576.

\bibitem{Mcm}
\newblock C. McMullen,
\newblock \textit{The Hausdorff dimension of general Sierpinski carpets},
\newblock Nagoya Math. J., \textbf{96} (1984), 1--9.

\bibitem{Ro}
\newblock V. A. Rohlin,
\newblock \textit{On the fundamental ideas of measure theory},
\newblock Amer. Math. Soc. Transl., \textbf{71} (1952), 55 pp.

\bibitem{Rue}
\newblock D. Ruelle,
\newblock \textit{Statistical mechanics on a compact set with $Z^v$ action satisfying expansiveness and specification},
\newblock Trans. Amer. Math. Soc., \textbf{185} (1973), 237--251.

\bibitem{Tsu}
\newblock M. Tsukamoto,
\newblock \textit{New approach to weighted topological entropy and pressure},
\newblock Ergod. Theory Dyn. Syst., \textbf{43} (2023), 1004--1034.

\bibitem{Pw1}
\newblock P. Walters,
\newblock \textit{A variational principle for the pressure of continuous transformations},
\newblock Am. J. Math., \textbf{97} (1975), 937--971.

\bibitem{Pw}
\newblock P. Walters,
\newblock \textit{An introduction to ergodic theory},
\newblock Springer, 1982.

\bibitem{WY}
\newblock T. Wang and Y. Huang,
\newblock \textit{Weighted topological and measure-theoretic entropy},
\newblock Discrete Contin. Dyn. Syst., \textbf{39} (2019), 3941--3967.

\bibitem{YK}
\newblock K. Yan,
\newblock \textit{Conditional entropy and fiber entropy for amenable group actions},
\newblock J. Differential Equations, \textbf{259} (2015), 3004--3031.

\bibitem{GZ}
\newblock G. Zhang,
\newblock \textit{Relative entropy, asymptotic pairs and chaos},
\newblock J. London Math. Soc., \textbf{73} (2006), 157--172.



\end{thebibliography}
\end{document}